\documentclass[a4paper,11pt]{amsart}

\usepackage{amsmath}
\usepackage{amssymb}
\usepackage{amsthm}
\usepackage[abbrev,alphabetic]{amsrefs}
\usepackage{amscd}
\usepackage[dvipdfmx]{graphicx}
\usepackage{hyperref}

\usepackage{amsthm}

\usepackage[normalem]{ulem} 
\newcommand\redout{\bgroup\markoverwith
{\textcolor{red}{\rule[.5ex]{2pt}{0.4pt}}}\ULon}
\newcommand\blueout{\bgroup\markoverwith
{\textcolor{blue}{\rule[.5ex]{2pt}{0.4pt}}}\ULon}

\subjclass[2010]{Primary 14J45; Secondary 14E30}

\keywords{rational points, Witt vectors, Nadel vanishing theorem}

\newtheorem{thm}{Theorem}[section]
\newtheorem{lem}[thm]{Lemma}

\newtheorem{cor}[thm]{Corollary}
\newtheorem{prop}[thm]{Proposition}

\newtheorem{claim}[thm]{Claim}

\newtheorem{step}{Step}

\theoremstyle{definition}

\newtheorem{dfn}[thm]{Definition}

\theoremstyle{remark}

\newtheorem{rem}[thm]{Remark}
\newtheorem*{ackn}{Acknowledgements}

\newcommand{\MO}{\mathcal{O}}
\newcommand{\F}{\mathbb{F}}
\newcommand{\Q}{\mathbb{Q}}
\newcommand{\R}{\mathbb{R}}
\newcommand{\Z}{\mathbb{Z}}

\newcommand{\mfm}{\mathfrak{m}}
\newcommand{\mfn}{\mathfrak{n}}

\newcommand{\red}[0]{{\operatorname{red}}}
\newcommand{\Ex}[0]{{\operatorname{Ex}}}

\DeclareMathOperator{\Supp}{Supp}

\DeclareMathOperator{\Spec}{Spec}

\DeclareMathOperator{\Nklt}{Nklt}

\makeatletter
  
  \@addtoreset{equation}{thm}
  \makeatother

\title[A Witt Nadel vanishing theorem for threefolds]
{A Witt Nadel vanishing theorem for threefolds}

\author{Yusuke Nakamura}

\address{Graduate School of Mathematical Sciences, 
the University of Tokyo, 3-8-1 Komaba, Meguro-ku, Tokyo 153-8914, Japan.}

\email{nakamura@ms.u-tokyo.ac.jp}

\author{Hiromu Tanaka}\address{Graduate School of Mathematical Sciences, 
the University of Tokyo, 3-8-1 Komaba, Meguro-ku, Tokyo 153-8914, Japan.}

\email{tanaka@ms.u-tokyo.ac.jp}

\begin{document}

\maketitle

\begin{abstract}
In this paper, we establish a vanishing theorem of Nadel type 
for the Witt multiplier ideals on threefolds 
over perfect fields of characteristic larger than five. 
As an application, if a projective normal threefold over $\mathbb F_q$ is not klt and 
its canonical divisor is anti-ample, 
then the number of the rational points on the klt-locus is divisible by $q$. 
\end{abstract}

\tableofcontents

\section{Introduction}
One of useful tools in complex algebraic geometry 
is the Kodaira vanishing theorem, 
which is generalised to the Kawamata--Viehweg vanishing theorem 
and the Nadel vanishing theorem. 
For instance, these vanishing theorems yield the following consequences: 
\begin{enumerate}
\item[$(1)_0$] 
If $X$ is a smooth projective variety over $\mathbb C$ such that $-K_X$ is ample, 
then $H^i(X, \MO_X)=0$ for $i>0$. 
\item[$(2)_0$]  
If $(X, \Delta)$ is a projective klt pair over $\mathbb C$ such that $-(K_X+\Delta)$ is ample, then $H^i(X, \MO_X)=0$ for $i>0$. 
\item[$(3)_0$]  
If $(X, \Delta)$ is a projective log pair over $\mathbb C$ such that $-(K_X+\Delta)$ is ample, then $H^i \left( X, \mathcal J \left( X, \Delta \right) \right)=0$ for $i>0$, 
where $\mathcal J(X, \Delta)$ denotes the multiplier ideal sheaf of $(X, \Delta)$. 
\end{enumerate}
Indeed, $(1)_0$, $(2)_0$ and $(3)_0$ follow from the Kodaira, Kawamata--Viehweg, and Nadel vanishing theorem respectively 
(cf.\ \cite[Theorem 1-2-5]{KMM87}, \cite[Corollary 2.68]{KM98}, 
\cite[Corollary 9.4.15]{Laz2}).

Although the Kodaira vanishing is known to fail in positive characteristic (cf.\ \cite{Ray78}), 
similar vanishing still holds in positive characteristic 
in terms of Witt vectors.  
The first result in this direction is given by Esnault in \cite{Esn03}. 
\begin{enumerate}
\item[$(1)_p$] 
If $X$ is a smooth projective variety over a perfect field of characteristic $p>0$ such that $-K_X$ is ample, then $H^i \left( X, W\MO_{X, \Q} \right)=0$ for $i>0$. 
\end{enumerate}
Then it is natural to seek a positive-characteristic analogue of $(2)_0$. 
Indeed, this is partially established by Gongyo and the authors \cite{GNT}. 
\begin{enumerate}
\item[$(2)_p$]  
If $(X, \Delta)$ is a projective klt pair over a perfect field of characteristic $p>5$ such that $-(K_X+\Delta)$ is ample and $\dim X \leq 3$, then 
$H^i \left( X, W\MO_{X, \Q} \right)=0$ for $i>0$. 
\end{enumerate}

The main theorem of this paper is a positive-characteristic analogue of $(3)_0$
for the three-dimensional case of characteristic $p>5$. 
Furthermore, we treat a relative setting as follows. 

\begin{thm}[{$=$\ Theorem \ref{theorem:NWV}}]\label{intro-NWV}
Let $k$ be a perfect field of characteristic $p > 5$. 
Let $(X, \Delta)$ be a three-dimensional log pair over $k$ and 
let $f:X \to Z$ be a projective $k$-morphism to a quasi-projective $k$-scheme $Z$. 
Assume that $-(K_X + \Delta)$ is $f$-nef and $f$-big. 
Then, the equation 
\[
R^if_* \left( WI_{\Nklt (X, \Delta), \Q} \right) = 0
\]
 holds for $i > 0$, 
where $\Nklt (X, \Delta)$ denotes the reduced closed subscheme of $X$ 
consisting of the non-klt points of $(X, \Delta)$ and 
$I_{\Nklt (X, \Delta)}$ is the coherent ideal sheaf on $X$ corresponding to $\Nklt (X, \Delta)$ (cf.\ Remark \ref{r-ideal-radical}). 
\end{thm}

As a consequence of Theorem \ref{intro-NWV}, 
we obtain the Koll\'ar--Shokurov connectedness theorem. 

\begin{thm}[$=$\ Theorem \ref{theorem:conn}]\label{intro-conn}
Let $k$ be a perfect field of characteristic $p > 5$. 
Let $(X, \Delta)$ be a three-dimensional log pair over $k$ and 
let $f:X \to Z$ be a projective $k$-morphism to a quasi-projective scheme $Z$ over $k$ 
such that $f_*\MO_X=\MO_Z$. 
Assume that $-(K_X + \Delta)$ is $f$-nef and $f$-big. 
If $\Nklt (X, \Delta)$ denotes the reduced closed subscheme of $X$ 
consisting of the non-klt points of $(X, \Delta)$ and let $g:\Nklt (X, \Delta) \to f \left( \Nklt (X, \Delta) \right)$ be the induced morphism, 
then the fibre $g^{-1}(z)$ over an arbitrary point $z$ of $f \left( \Nklt \left( X, \Delta \right) \right)$ 
is geometrically connected over the residue field $k(z)$ at $z$. 
\end{thm}

\noindent
We note that Birkar proved a weaker version of this theorem in the case when $f$ is birational and the coefficients of 
$\Delta$ are at most one (\cite[Theorem 1.8]{Bir16}).

Also, we have applications of Theorem \ref{intro-NWV} 
to rational points on varieties over finite fields. 
The starting-point is the following theorem by Esnault \cite{Esn03}, 
which is a consequence of $(1)_p$ and a Lefschetz trace formula 
for $W\MO_{X, \Q}$.

\begin{enumerate}
\item[$(1)'_p$] 
If $X$ is a geometrically connected smooth projective variety 
over a finite field $k$ such that $-K_X$ is ample, 
then $\# X(k) \equiv 1 \pmod {\# k}$. 
\end{enumerate}

\noindent
In \cite{GNT}, Gongyo and the authors prove that the same formula still holds for Fano threefolds with klt singularities.

\begin{enumerate}
\item[$(2)'_p$] 
If $(X, \Delta)$ is a three-dimensional geometrically connected projective klt pair over a finite field 
$k$ of characteristic $p>5$ such that $-(K_X+\Delta)$ is ample, 
then $\# X(k) \equiv 1 \pmod {\# k}$. 
\end{enumerate}

Then it is natural to seek an application of Theorem \ref{intro-NWV} 
to the number of the rational points on a non-klt Fano threefold. 
In this direction, we show that the number of the rational points on the klt-locus is divisible by $\# k$. 

\begin{thm}[$=$\ Corollary \ref{cor:klt_locus}]\label{intro-rat0}
Let $(X, \Delta)$ be a three-dimensional geometrically connected projective log pair over a finite field $k$ of characteristic $p>5$. 
Assume that $-(K_X + \Delta)$ is nef and big and that $(X, \Delta)$ is not klt. 
Then, the congruence 
\[
\# X (k) \equiv \# V (k)  \mod \# k
\]
holds, where $V$ denotes the closed subset of $X$ consisting of the non-klt points of $(X, \Delta)$. 
\end{thm}

\noindent
An interesting point is that this theorem is not true 
if we drop the assumption that $(X, \Delta)$ is not klt (cf.\ $(2)'_p$). 
On the other hand, the following theorem gives a common generalisation of $(2)'_p$ and Theorem \ref{intro-rat0}.  
Moreover, we treat a relative setting.

\begin{thm}[$=$\ Theorem \ref{t-rat-pt1}]\label{intro-rat-pt1}
Let $(X, \Delta)$ be a three-dimensional log pair over 
a finite field $k$ of characteristic $p>5$. 
Let $f:X \to Y$ be a projective $k$-morphism to a quasi-projective $k$-scheme $Y$ 
such that $f_*\MO_X=\MO_Y$. 
Assume that $-(K_X+\Delta)$ is $f$-nef and $f$-big. 
Then, the congruence 
\[
\# X(k) -  \# V (k) \equiv \# Y(k) - \# f(V) (k)  \mod \# k
\]
holds, where $V$ denotes the closed subset of $X$ consisting of the non-klt points of $(X, \Delta)$. 
\end{thm}

Furthermore, we shall show that some hypersurfaces $D$ on smooth Fano threefolds 
contain rational points even if $D$ is not klt. 
It can be seen as a variation of the Ax--Katz theorem 
(cf.\ \cite{Ax64}, \cite{Kat71}).

\begin{thm}[cf.\ Theorem \ref{t-rat-pt2}]\label{intro-rat-pt2}
Let $X$ be a three-dimensional projective geometrically connected variety with klt singularities 
over a finite field $k$ of characteristic $p>5$. 
Let $D$ be a nonzero effective $\mathbb{Q}$-Cartier Weil divisor on $X$. 
Assume that 
\begin{enumerate}
\item 
$-K_X$ is ample, and 
\item 
$-(K_X+D)$ is ample. 
\end{enumerate}
Then, the congruence 
\[
\# D(k) \equiv 1 \mod \# k
\]
holds. In particular, there exists a $k$-rational point on $D$. 
\end{thm}

\noindent 
In Theorem \ref{t-rat-pt2}, we work on a more general setting. 

\subsection{Description of the proof}\label{ss1-intro}
We now overview some of the ideas of the proof of Theorem \ref{intro-NWV}. 
In the following, we work over a perfect field $k$ of characteristic $p>5$. 
Roughly speaking, the argument consists of two steps: 
\begin{enumerate}
\item[(A)]  
We prove the $W\MO$-vanishing for log Fano contractions, i.e.\ Theorem \ref{t-rel-van} (Section \ref{s-WO-MFS}). 
\item[(B)]  
Using Theorem \ref{t-rel-van}, we prove Theorem \ref{intro-NWV} 
(Section \ref{s-nadel}). 
\end{enumerate}

(A) 
Now, let us give an overview of how to prove Theorem \ref{t-rel-van}. 
Given a three-dimensional klt pair $(X, \Delta)$ and a projective morphism 
$f:X \to Z$ such that $-(K_X+\Delta)$ is $f$-ample and $f_*\MO_X=\MO_Z$, 
we want to prove that $R^if_* \left( W\MO_{X, \Q} \right)=0$ for any $i>0$. 
We further divide the proof of Theorem \ref{t-rel-van} into the four cases depending 
on the dimension of $Z$.  
Since the cases $\dim Z=0$ and $\dim Z=3$ have been settled already 
in \cite{GNT}, 
it is enough to assume either 
\begin{enumerate}
\item[(A1)]  $\dim Z=1$ (Subsection \ref{ss-dP-fib}), or
\item[(A2)]   $\dim Z=2$ (Subsection \ref{ss-conic-bdl}). 
\end{enumerate}

(A1) 
We first treat the case when $\dim Z=1$. 
In this case, the generic fibre $X_{K(Z)}$ of $f$ is a log del Pezzo surface over an imperfect field. 
One of significant steps is to show that $\big( X_{K(Z)} \times_{K(Z)} \overline{K(Z)} \big)_{\red}$ 
is a rational surface (Proposition \ref{p-klt-dP}). 
Indeed, this result enables us to use 
a result by Chatzistamatiou--R{\"u}lling (Theorem \ref{thm:CR481}) 
after taking suitable purely inseparable covers of $X$ and $Z$ (cf.\ Proposition \ref{p-dP-fib}), which in turn implies what we want. 

(A2) 
We now treat the case when $\dim Z=2$. 
The crucial part of this case is to prove that 
$Z$ has $W\MO$-rational singularities (Theorem \ref{t-conic-base}). 
To this end, we first reduce the problem to the case when $k=\overline{\F}_p$. 
Assume $k=\overline{\F}_p$.  
In order to prove that $Z$ has $W\MO$-rational singularities, 
we compute, for sufficiently divisible $e \in \Z_{>0}$, 
the numbers of $\F_{p^e}$-rational points on the models $X'$ and $Z'$ over $\F_{p^e}$ of $X$ and $Z$ respectively.  
(cf.\ Step \ref{step-fp-bar} in the proof of Theorem \ref{t-conic-base}).

(B) 
We now overview how to prove Theorem \ref{intro-NWV}. 
For simplicity, we treat only the case when $Z=\Spec\,k$ 
and $k$ is an algebraically closed field. 
Taking a dlt modification of $(X, \Delta)$ (Proposition \ref{p-dlt-modif}), 
we may assume that $X$ is $\Q$-factorial, $\left( X, \Delta^{\wedge 1} \right)$ is dlt  (for definition of $(-)^{\wedge 1}$, 
see Subsection \ref{ss-notation}), and $-(K_X+\Delta)$ is ample (cf.\ Lemma \ref{l-dlt-modif2}). 
By the ampleness of $-(K_X+\Delta)$, we can find an effective $\mathbb{R}$-divisor $\Omega$ on $X$ such that 
\begin{enumerate}
\item $\left( X, \Omega^{\wedge 1} \right)$ is dlt, 
\item $K_X+\Omega \sim_{\mathbb{R}} 0$, 
\item $\Omega$ is big, 
\item $\Supp \Omega^{\geq 1}=\Supp \Omega^{>1}$, and  
\item $\Supp \Nklt(X, \Omega)=\Supp \Nklt(X, \Delta)$. 
\end{enumerate}
Then it suffices to prove the vanishing $H^i \left( X, WI_{\Nklt \left( X, \Omega \right)} \right)=0$ for $i>0$. 
Furthermore, we may assume that $\Supp\,\Omega^{>1} \neq \emptyset$ since the assertion is nothing but \cite[Theorem 1.3]{GNT} 
when $(X, \Delta)$ is klt. 

The first step is to run a $\left( K_X+\Omega^{\wedge 1} \right)$-MMP 
in order to reduce the problem to the end result (cf.\ Proposition \ref{prop:inv_mmp}). 
In Proposition \ref{prop:inv_mmp}, it is proved that the cohomologies are preserved under this MMP. 

Replacing $X$ by the end result, let us assume that $X$ itself is the end result 
of this MMP. 
By (2) and $\Supp\,\Omega^{>1} \neq \emptyset$, $X$ has a $\left( K_X+\Omega^{\wedge 1} \right)$-Mori fibre space structure $g:X \to W$. 
Then the problem is reduced to vanishing of cohomologies for dlt Mori fibre spaces 
(Lemma \ref{lem:MFS}). 
By induction on the number of the irreducible components of 
$\llcorner \Xi\lrcorner$ for $\Xi:=\Omega^{\wedge 1}$, Lemma \ref{lem:MFS} is proved by using 
the $W\MO$-vanishing for klt Mori fibre spaces (Theorem \ref{t-rel-van}).

\begin{ackn} 
We would like to thank Professors 
H{\'e}l{\`e}ne Esnault, Yoshinori Gongyo, Shunsuke Takagi and Chenyang Xu for many useful discussions. 
The authors also thank the referees for reading the manuscript carefully, 
suggesting several improvements, and pointing out mistakes. 
The first author is partially supported by JSPS KAKENHI Grant Number 18K13384. 
The second author is partially supported by EPSRC and JSPS KAKENHI Grant Number 18K13386. 
\end{ackn}

\section{Preliminaries}

\subsection{Notation}\label{ss-notation}

In this subsection, we summarise notation used in this paper. 

\begin{itemize}
\item We will freely use the notation and terminology in \cite{Har77} 
and \cite{Kol13}. 

\item For a scheme $X$, its {\em reduced structure} $X_{\red}$ 
is the reduced closed subscheme of $X$ such that the induced closed immersion $X_{\red} \to X$ is surjective. 

\item 
A morphism $f:X \to Y$ of schemes {\em has connected fibres} 
if $X \times_Y \Spec\,L$ is either empty or connected for any field $L$ and any morphism 
$\Spec\,L \to Y$. 

\item For an integral scheme $X$, 
we define the {\em function field} $K(X)$ of $X$ 
as $\MO_{X, \xi}$ for the generic point $\xi$ of $X$. 

\item For a field $k$, 
we say that $X$ is a {\em variety over} $k$ or a $k$-{\em variety} if 
$X$ is an integral scheme that is separated and of finite type over $k$. 
We say that $X$ is a {\em curve} over $k$ or a $k$-{\em curve} 
(resp.\ a {\em surface} over $k$ or a $k$-{\em surface}, 
resp.\ a {\em threefold} over $k$) 
if $X$ is a $k$-variety of dimension one (resp.\ two, resp.\ three). 

\item For a field $k$, let $\overline k$ be an algebraic closure of $k$. 
If $k$ is of characteristic $p>0$, 
then we set $k^{1/p^{\infty}}:=\bigcup_{e=0}^{\infty} k^{1/p^e}
=\bigcup_{e=0}^{\infty} \left\{ x \in \overline k\, \big|\, x^{p^e} \in k \right\}$. 

\item Let $f:X \to Y$ be a projective morphism of noetherian schemes. 
Let $M$ be an $\R$-Cartier $\R$-divisor $M$ on $X$. 
We say that $M$ is $f$-{\em ample}
if we can write $M=\sum_{i=1}^r a_iM_i$ 
for some $r \geq 1$, positive real numbers $a_i$ and 
$f$-ample Cartier divisors $M_i$. 
We say that $M$ is $f$-{\em big} if we can write $M=A+E$ 
for some $f$-ample $\R$-Cartier $\R$-divisor $A$ and effective $\R$-divisor $E$. 
We can define $f$-nef $\R$-divisors in the same way as in \cite[Definition 1.4]{Kol13}. 
We say that $M$ is \textit{$f$-numerically-trivial}, denoted by $M \equiv_f 0$, 
if both $M$ and $-M$ are $f$-nef. 

\item Let $\Delta = \sum r_i D_i$ be an $\mathbb{R}$-divisor where $D_i$ are distinct prime divisors. 
We define $\Delta ^{\ge 1} := \sum_{r_i \ge 1} r_i  D_i$ and $\Delta ^{\wedge 1} := \sum r_i ' D_i$ where $r_i ' := \min \{ r_i, 1  \}$. 
We also define $\Delta ^{> 1}$ and $\Delta ^{< 1}$ similarly. 
Moreover, we denote $\{ \Delta \} = \Delta - \lfloor \Delta \rfloor$. 

\item A \textit{sub-log pair} $(X, \Delta)$ over a filed $k$ consists of a normal variety $X$ over $k$ and an 
$\mathbb{R}$-divisor $\Delta$ such that 
$K_X + \Delta$ is $\mathbb{R}$-Cartier. 
A \textit{log pair} $(X, \Delta)$ is a sub-log pair such that 
$\Delta \geq 0$.

\item 
For a closed subscheme $V$ of a scheme $X$, 
we denote by $I_V$ the quasi-coherent ideal sheaf corresponding to $V$. 
For an effective $\mathbb{R}$-divisor $D$ on a normal variety $X$ over a field, 
we denote by $I_D:=I_{\mathcal D}$, 
where $\mathcal D$ denotes the closed subscheme of $X$ 
corresponding to the coherent ideal sheaf $\MO_X \left( -\ulcorner D\urcorner \right)$ on $X$. 

\item 
For terminology on derived category, 
we refer to \cite{Wei94}. 
Especially, for morphisms $X \xrightarrow{f} Y \xrightarrow{g} Z$ of $\Z/p\Z$-schemes 
and a $W\MO_X$-module $M$, 
we shall frequently use the isomorphism: $Rg_* \circ Rf_* (M) \simeq R(g \circ f)_* (M)$ 
(cf. \cite[Corollary 10.8.10]{Wei94}). 
\end{itemize}

\subsection{Results on minimal model program}\label{s-mmp-sing}

\begin{dfn}\label{d-non-klt}
Let $k$ be a field. 
\begin{enumerate}
\item 
We say that $(X, \Delta)$ is a {\em sub-klt} pair over $k$ if 
$(X, \Delta)$ is a sub-log pair over $k$ such that $(X, \Delta)$ is klt in the sense of \cite[Definition 2.8]{Kol13}. 
Given a point $x$ of $X$, 
we say that $(X, \Delta)$ is {\em sub-klt around} $x$ if there exists an open neighbourhood $X'$ of $x \in X$ such that $\left( X', \Delta|_{X'} \right)$ is sub-klt. 
\item 
We say that $(X, \Delta)$ is {\em klt} (resp.\ {\em log canonical}) if 
$(X, \Delta)$ is a log pair such that $(X, \Delta)$ is klt (resp.\ log canonical) in the sense of \cite[Definition 2.8]{Kol13}. 
\item 
Given a point $x$ of $X$, 
we say that $x$ is a {\em non-klt point} of $(X, \Delta)$ if $(X, \Delta)$ is not sub-klt around $x$. 
We define $\Nklt(X, \Delta)$, called the {\em non-klt locus of $(X, \Delta)$}, 
as the subset of $X$ that consists of all the non-klt points. 
Note that $\Nklt(X, \Delta)$ is a closed subset of $X$, as its complement is an open subset of $X$ by definition. 
We equip $\Nklt(X, \Delta)$ with the reduced scheme structure. 
\end{enumerate}
\end{dfn}

\begin{rem}\label{r-ideal-radical}
For coherent ideal sheaves $I, J \subset \mathcal{O}_X$ with $\sqrt{I} = \sqrt{J}$, 
it follows that $WI_{\mathbb{Q}} = WJ_{\mathbb{Q}}$ (cf.\ \cite[Proposition 2.1]{BBE07}). 
Hence we need not to care about the scheme structure of $\Nklt (X, \Delta)$ when we consider $WI_{\Nklt (X, \Delta), \mathbb{Q}}$. 
By the same reason, 
if $X$ is a non-reduced noetherian scheme and $j:X_{\red} \to X$ 
denotes the closed immersion from its reduced structure $X_{\red}$, 
then the induced homomorphism $W\MO_{X, \Q} \to j_*W\MO_{X_{\red}, \Q}$ 
is an isomorphism. See Lemma \ref{l-GNT2.22} for a generalisation. 
\end{rem}

\begin{dfn}
A log pair $(X, \Delta)$ is called \textit{dlt} if the coefficients of $\Delta$ are at most one and 
there exists a log resolution $g: Y \to X$ of $(X, \Delta)$ with the condition that 
$a_E (X, \Delta) > 0$ holds for any $g$-exceptional prime divisor $E$ on $Y$. 
\end{dfn}

\begin{dfn}\label{d-Fano-type}
Given a field $k$ and a projective $k$-morphism $f:X \to Z$ 
from a normal $k$-variety $X$ to a quasi-projective $k$-scheme $Z$, 
we say that $X$ is {\em of Fano type over $Z$} 
if there exists an effective $\mathbb{R}$-divisor $\Delta$ on $X$ such that 
 $(X, \Delta)$ is klt and $-(K_X+\Delta)$ is $f$-nef and $f$-big. 
\end{dfn}

\begin{dfn}
Given a field $k$, a log pair $(X, \Delta)$ over $k$, 
and projective $k$-morphisms $X \xrightarrow{f_1} Z_1 \to Z_2$ 
of quasi-projective $k$-schemes, 
we say that $f_1:X \to Z_1$ is a $(K_X+\Delta)$-{\em Mori fibre space over} $Z_2$ 
if $\dim X>\dim Z_1$, $\left( f_1 \right)_*\MO_X=\MO_{Z_1}$, $\rho(X/Z_1)=1$ and $-(K_X+\Delta)$ is $f_1$-ample. 
If $Z_2=\Spec\,k$, then $f_1:X \to Z_1$ is simply called a $(K_X+\Delta)$-{\em Mori fibre space}. 
\end{dfn}

\begin{lem}\label{l-nklt-crep}
Let $k$ be a field and let $f:X \to Y$ be a projective birational $k$-morphism of normal varieties over $k$. 
Let $(Y, \Delta_Y)$ be a sub-log pair and 
let $\Delta$ be the $\R$-divisor defined by 
$K_X+\Delta = f^*(K_Y+\Delta_Y)$. 
Then the following hold. 
\begin{enumerate}
\item Let $x$ be a closed point of $X$. 
If $x$ is a non-klt point of $(X, \Delta)$, then $f(x)$ is a non-klt point of $(Y, \Delta_Y)$. 
\item Let $y$ be a closed point of $Y$. 
If $y$ is a non-klt point of $\left( Y, \Delta_Y \right)$, then there exists a closed point $x$ of $X$ 
such that $f(x)=y$ and $x$ is a non-klt point of $(X, \Delta)$. 
\end{enumerate}
In particular, there exists a commutative diagram consisting of projective morphisms: 
$$\begin{CD}
\Nklt(X, \Delta) @>>> X\\
@VVf'V @VVfV\\
\Nklt(Y, \Delta_Y) @>>> Y,\\
\end{CD}$$
where the horizontal arrows are the induced closed immersions and 
$f'$ is a projective surjective morphism. 
In particular,  $f \left( \Nklt \left ( X, \Delta \right) \right)=\Nklt(Y, \Delta_Y)$. 
\end{lem}

\begin{proof}
Both of the assertions follow from the fact that $\left( U, \Delta |_U \right)$ is sub-klt if and only if $\left( f^{-1}(U), \Delta_Y |_{f^{-1}(U)} \right)$ is sub-klt 
for any open subset $U \subset X$ (cf.\ \cite[Lemma 2.30]{KM98}). 
\end{proof}

\begin{prop}\label{p-nklt-non-negative}
Let $(X, \Delta)$ be a $\Q$-factorial sub-log pair over a field such that 
$\bigl( X,  ( \Delta^{> 0} ) ^{< 1} \bigr)$ is klt. 
Then it holds that 
$$\Nklt(X, \Delta)=\Nklt \left( X, \Delta^{> 0} \right)=\Supp \Delta^{\geq 1}.$$
\end{prop}

\begin{proof}
The second equality follows from the fact that 
$\bigl( X, ( \Delta^{> 0} ) ^{< 1} \bigr)$ is klt. 
Clearly the inclusion $\Nklt(X, \Delta) \subset \Nklt \left( X, \Delta^{> 0} \right)$ holds. 
It suffices to prove the opposite one. 
Let $D$ be a prime divisor contained in $\Supp \Delta^{\geq 1}$. 
Since $\Nklt(X, \Delta)$ is a closed subset containing general closed points of $D$, 
we have that $D \subset \Nklt(X, \Delta)$. 
In particular, $\Supp \Delta^{\geq 1} \subset \Nklt(X, \Delta)$. 
\end{proof}

\begin{lem}\label{l-perturb-dlt}
Let $\mathbb{K} \in \{ \mathbb{Q}, \mathbb{R} \}$. Let $k$ be a perfect field of characteristic $p>0$. 
Let $(X, \Delta)$ be a quasi-projective dlt pair over $k$ with $\dim X \leq 3$. 
Let $A$ be an ample $\mathbb{K}$-Cartier $\mathbb{K}$-divisor on $X$. 
Then there exists an effective $\mathbb{K}$-Cartier $\mathbb{K}$-divisor $A'$ on $X$ 
such that $A \sim_{\mathbb{K}} A'$ and $\left( X, \Delta+A' \right)$ is dlt. 
\end{lem}

\begin{proof}
If $k$ is an infinite field, 
then the proof of \cite[Lemma 9.2]{Bir16} works without any changes. 
Assume that $k$ is a finite field. 
Thanks to \cite[Theorem 1.1]{Poo04}, we can still make use of Bertini's theorem. 
Hence, we can apply the same argument as in \cite[Lemma 9.2]{Bir16}. 
\end{proof}

The existence of a minimal model program is known for log canonical threefolds. 
For terminology appearing in the following theorem, 
we refer to \cite[Subsection 2.4]{HNT}. 

\begin{thm}\label{t-lc-MMP}
Let $k$ be a perfect field of characteristic $p>5$. 
Let $(X, \Delta)$ be a three-dimensional $\Q$-factorial log canonical pair over $k$, 
where $\Delta$ is an $\R$-divisor. 
Let $f:X \to Z$ be a projective $k$-morphism to a quasi-projective $k$-scheme $Z$. 
Then there exists a $(K_X+\Delta)$-MMP over $Z$ that terminates. 
In other words, 
there is a sequence of birational maps of three-dimensional normal varieties:  
\[
X=:X_0 \overset{\varphi_0}{\dashrightarrow} X_1 \overset{\varphi_1}{\dashrightarrow} \cdots \overset{\varphi_{\ell-1}}{\dashrightarrow} X_{\ell}
\]
such that if $\Delta_i$ denotes the proper transform of $\Delta$ on $X_i$, then 
the following properties hold:  
\begin{enumerate}
\item 
For any $i \in \{0, \ldots, \ell\}$, 
$\left( X_i, \Delta_i \right)$ is a $\Q$-factorial log canonical pair which is projective over $Z$.
\item 
For any $i \in \{0, \ldots, \ell-1\}$, 
$\varphi_i:X_i \dashrightarrow X_{i+1}$ is either a $\left( K_{X_i}+\Delta_i \right)$-divisorial contraction over $Z$ or a $\left( K_{X_i}+\Delta_i \right)$-flip over $Z$. 
\item 
If $K_X+\Delta$ is pseudo-effective over $Z$, then $K_{X_{\ell}}+\Delta_{\ell}$ is nef over $Z$. 
\item 
If $K_X+\Delta$ is not pseudo-effective over $Z$, then 
there exists a $\left( K_{X_{\ell}}+\Delta_{\ell} \right)$-Mori fibre space $X_{\ell} \to Y$ over $Z$. 
\end{enumerate}
\end{thm}

\begin{proof}
See \cite[Theorem 1.1]{HNT}. 
\end{proof}

\begin{prop}\label{p-dlt-modif} 
Let $(X, \Delta)$ be a three-dimensional quasi-projective log pair over a perfect field $k$ of characteristic $p > 5$. 
Then there exists a projective birational morphism $f:Y \to X$ that satisfies the following conditions: 
\begin{enumerate}
\item $a_F (X, \Delta) \le 0$ holds for any $f$-exceptional prime divisor $F$. 
\item $\left( Y, \Delta _Y ^{\wedge 1} \right)$ is a $\Q$-factorial dlt pair,  
where $\Delta _Y$ is the $\mathbb{R}$-divisor defined  by $K_Y + \Delta _Y = f^* (K_X + \Delta)$ 
(see Subsection \ref{ss-notation} for the definition of $\Delta _Y ^{\wedge 1}$). 
\item 
$\Nklt \left( Y, \Delta _{Y} \right) = f ^{-1} \left( \Nklt \left( X, \Delta \right) \right)$ holds. 
\end{enumerate} 
\end{prop}

\begin{proof}
See \cite[Proposition 3.5]{HNT}. 
\end{proof}

For later use, we establish the following result on plt centres.

\begin{prop}\label{p-plt-centre} 
Let $(X, \Delta)$ be a three-dimensional plt pair over a perfect field $k$ of characteristic $p > 5$. 
Set $S:=\llcorner \Delta \lrcorner$. 
Then the normalisation $\nu:S^N \to S$ of $S$ is a universal homeomorphism. 
\end{prop}

\begin{proof}
Let $f:Y \to X$ and  $\Delta _Y$ be as in Proposition \ref{p-dlt-modif}. 
Since $(X, \Delta)$ is plt, it follows from Proposition \ref{p-dlt-modif}(2) that $(Y, \Delta_Y)$ is also plt. 
For $S_Y:=\llcorner \Delta_Y \lrcorner$, 
we have that $f^{-1}(S)=f ^{-1} \left( \Nklt (X, \Delta) \right) =\Nklt \left( Y, \Delta _{Y} \right) = S_Y$, 
where the second equality holds by Proposition \ref{p-dlt-modif}(3). 
Hence, the induced morphism $S_Y \to S$ has connected fibres. 
Since $Y$ is $\Q$-factorial and $\left( Y, \Delta_Y \right)$ is plt, 
$S_Y$ is normal (cf. \cite[Theorem 2.11]{GNT}). 
Therefore, $S_Y \to S$ factors through the normalisation $\nu:S^N \to S$: 
\[
S_Y \to S^N \xrightarrow{\nu} S.
\] 
Then, $\nu$ is a finite morphism and has connected fibres, hence $\nu$ is a universal homeomorphism. 
\end{proof}

\subsection{Connectedness theorem for the birational case}

The purpose of this subsection is to establish 
the three-dimensional Koll\'ar--Shokurov connectedness theorem 
for the birational case (Theorem \ref{t-birat-connected2}). 
A key result is Proposition \ref{p-birat-connected}. 
To prove this proposition, we show Lemma \ref{l-strong-Birkar} (cf. \cite[Theorem 1.8]{Bir16}). 
We start with the following auxiliary result.

\begin{lem}\label{l-topology}
Let $X$ be a noetherian topological space (for definition, see \cite[page 5]{Har77}).  
Let $F$ be a closed subset of $X$ 
and let $\{N_i\}_{i \in I}$ be a set of closed subsets of $X$, 
where $I$ is a finite set. 
Set $N:=\bigcup_{i \in I} N_i$. 
Assume that the following hold. 
\begin{enumerate}
\item $N_i \cap F$ is connected for any $i \in I$ ($N_i \cap F$ is possibly empty). 
\item $N \cap U$ is connected for any sufficiently small open neighbourhood $U$ 
in $X$ of $F$. 
In other words, there exists an open subset $U_0$ of $X$ such that 
$F \subset U_0$ and if $U$ is an open subset of $X$ satisfying $F \subset U \subset U_0$, 
then $N \cap U$ is connected. 
\end{enumerate}
Then $N \cap F$ is connected. 
\end{lem}

\begin{proof}
We first reduce the problem to the case when $N_i \cap F \neq \emptyset$ for any $i \in I$. 
Set $I':=\{i \in I\,|\, N_i \cap F \neq \emptyset\}$ and 
\[
U'_0 := U_0 \setminus \Biggl(\bigcup_{i \in I \setminus I'} N_i \Biggr)
= U_0 \cap \Biggl(\bigcap_{i \in I \setminus I'} (X \setminus N_i) \Biggr),
\]
where $U_0$ is as in (2). 
Then we have $F \subset U'_0$. 
If $U$ is an open subset of $X$ such that $F \subset U \subset U'_0$, 
then (2) implies that $N \cap U$ is connected. 
Therefore, the problem can be reduced to the case when $N_i \cap F \neq \emptyset$ for any $i \in I$. 
In what follows, we assume that $N_i \cap F \neq \emptyset$ for any $i \in I$.

Take the decomposition into connected components of $N \cap F$: 
\[
N \cap F= \coprod_{j \in J} \Gamma_j.
\]
We also have $N \cap F= \bigcup_{i \in I} (N_i \cap F)$. We first show that 
\begin{enumerate}
\setcounter{enumi}{2}
\item for any $i \in I$, there exists an index $j_i \in J$ such that $N_i \cap \Gamma_{j_i} \neq \emptyset$ and $N_i \cap \Gamma_j = \emptyset$ for any $j \in J \setminus \{j_i\}$. 
In particular, it holds that $N_i \cap F= N_i \cap \Gamma_{j_i}$ for any $i \in I$. 
\end{enumerate}
Fix $i \in I$. 
We have 
\[
N_i \cap F= \coprod_{j \in J} \left( N_i \cap \Gamma_j \right). 
\]
Since $N_i \cap F$ is non-empty and connected, (3) holds.

For $j \in J$, set $I_j:=\{i \in I\,|\, N_i \cap \Gamma_j \neq \emptyset\}$ 
and $N_{I_j}:=\bigcup_{i \in I_j} N_i$. 
Then (3) implies that $I = \coprod_{j \in J} I_j$. 
In particular, we obtain $N= \bigcup_{j \in J} N_{I_j}$. 
Set 
\begin{equation}\label{e1-topology}
U_1 := \bigcap_{j, j' \in J, j\neq j'} X \setminus \bigl( N_{I_j} \cap N_{I_{j'}} \bigr), 
\end{equation}
which is an open subset of $X$. 

We now show that $F \subset U_1$. 
Pick $j, j' \in J$ such that $j \neq j'$. 
For $i \in I_j$ and $i' \in I_{j'}$, we have 
\[
N_{i} \cap N_{i'} \cap F = \left( N_{i} \cap F \right) \cap \left( N_{i'} \cap F \right) 
\subset \Gamma_{j_{i}} \cap \Gamma_{j_{i'}} = \Gamma_{j} \cap \Gamma_{j'}=\emptyset,
\] 
where the inclusion follows from (3). 
Hence, we have 
\[
N_{I_{j}} \cap N_{I_{j'}} \cap F=
\bigcup_{i \in I_{j}, i' \in I_{j'}} \left( N_{i} \cap N_{i'} \cap F \right) 
= \emptyset, 
\]
i.e. $F \subset X \setminus \bigl( N_{I_{j}} \cap N_{I_{j'}} \bigr)$. 
Hence, (\ref{e1-topology}) implies $F \subset U_1$.

Set $U:=U_0 \cap U_1$, where $U_0$ is as in (2). 
Since $F \subset U_0 \cap U_1 = U$, (2) implies that $N \cap U$ is connected. 
We have 
\[
N \cap U = \bigcup_{j \in J} \left( N_{I_j} \cap U \right) = \coprod_{j \in J} \left( N_{I_j} \cap U \right), 
\]
where the last equality follows from (\ref{e1-topology}). 
For $j \in J$, we have that 
\[
N_{I_j} \cap U 
\supset N_{I_j} \cap F \neq \emptyset. 
\] 
Since $N \cap U$ is connected, we obtain $|J|=1$, i.e. $N \cap F$ is connected. 
\end{proof}

\begin{lem}\label{l-strong-Birkar}
Let $k$ be an algebraically closed field of characteristic $p>5$. 
Let $(X, D)$ be a three-dimensional $\Q$-factorial dlt pair over $k$ 
and let $f:X \to Y$ be a projective birational $k$-morphism 
to a normal threefold $Y$ over $k$. 
If $f$ is either a $(K_X+D)$-divisorial contraction or 
a $(K_X+D)$-flipping contraction, then the induced morphism 
$\Nklt(X, D) \to Y$ has connected fibres.  
\end{lem}

\begin{proof}
Set $S:=\llcorner D\lrcorner$ and 
let $S=\sum_{i \in I} S_i$ be the irreducible decomposition. 
We have $\Nklt(X, D) = \bigcup_{i \in I} S_i$.

For any sufficiently small open neighbourhood $U$ in $X$  of $f^{-1}(y)$,  
it follows from \cite[Theorem 1.8]{Bir16} that $\Nklt(X, D) \cap U$ is connected. 
We apply Lemma \ref{l-topology} to $N_i := S_i$ and $F:=f^{-1}(y)$. 
Then it is enough to prove that $S_i \cap F$ is connected. 
Therefore, after perturbing coefficients of $D$, we may assume  that $S=S_i$, 
i.e. $(X, D)$ is plt.

From now on, we treat the case when $\llcorner D \lrcorner =S$ is a prime divisor. 
In this case, we have $\Nklt(X, D) =S$. 
If $S \subset \Ex(f)$, then $f$ is a $(K_X+D)$-divisorial contraction 
such that $S=\Ex(f)$. 
Then the assertion is clear because $f$ has connected fibres. 
Thus, we may assume that $S \not\subset \Ex(f)$. 
Since $-(K_X+D)$ is $f$-ample, 
there exists an effective $\R$-divisor $A$ on $X$ such that 
$A \sim_{\R, f} -(K_X+D)$ and $(X, D+A)$ is plt. 
Since $K_X+D+A \sim_{\R, f} 0$, we have $K_X+D+A=f^*(K_Y+D_Y+A_Y)$ 
for $D_Y:=f_*D$ and $A_Y:=f_*A$. 
In particular, $(Y, D_Y+A_Y)$ is plt. 
Then the induced morphism $g:S \to S_Y := f(S)$ has connected fibres by Proposition \ref{p-plt-centre}. 
Since $\Nklt(X, D) \cap f^{-1}(y) = S \cap f^{-1}(y) = g^{-1}(y)$ for any $y \in f(S) = S_Y$, 
$\Nklt(X, D) \to Y$ has connected fibres.  
\end{proof}

\begin{prop}\label{p-birat-connected}
Let $k$ be an algebraically closed field of characteristic $p>5$. 
Let $(V, \Delta)$ be a three-dimensional quasi-projective $\Q$-factorial log pair over $k$. 
Let $\varphi:U \to V$ be a log resolution of $(V, \Delta)$.  
Let $\Delta _U$ be the $\R$-divisor defined by $K_U+\Delta_U=\varphi^*(K_V+\Delta)$. 
Then the induced morphism $\Nklt(U, \Delta_U) \to V$ has connected fibres. 
\end{prop}

\begin{proof}
Let $F$ be the sum of the $\varphi$-exceptional prime divisors $F'$ on $U$
whose log discrepancies are positive: $a_{F'}(V, \Delta)>0$. 
Let $G$ be the sum of the $\varphi$-exceptional prime divisors $G'$ on $U$
whose log discrepancies are non-positive: $a_{G'}(V, \Delta) \le 0$. 
We set 
\[
D_U:= \varphi ^{-1} _* \Delta ^{\wedge 1} + (1 - \epsilon) F + G
\]
for a sufficiently small positive real number $\epsilon$. 
Then, it holds that 
\begin{enumerate}
\item $(U, D_U)$ is dlt, and 
\item $\Supp \Delta_U^{\geq 1} =\Supp D_U^{= 1}$. 
\end{enumerate}
By Theorem \ref{t-lc-MMP}, 
there is a $\left( K_U+ D_U \right)$-MMP over $V$ that terminates: 
\begin{equation}\label{e1-birat-connected}
U=:X_0 \dashrightarrow \cdots \dashrightarrow X_{\ell}.
\end{equation}
For any $i \in \{0, \ldots, \ell\}$, 
we define $\Delta_{X_i}$, $F_{X_i}$, $G_{X_i},$ and $D_{X_i}$ as the push-forwards of $\Delta_U$, $F$, $G,$ and $D_U$ on $X_i$, 
respectively.  
Then it holds that 
$K_{X_i}+\Delta_{X_i}=\psi_i^*(K_V+\Delta)$, 
where $\psi_i:X_i \to V$ denotes the induced morphism. 
Moreover, for any $i \in \{0, \cdots, \ell\}$, we get 
\begin{enumerate}
\item[(1)'] $(X_i, D_{X_i})$ is dlt, and 
\item[(2)'] $\Supp \Delta_{X_i}^{\geq 1} = \Supp D_{X_i}^{= 1}$. 
\end{enumerate}

\setcounter{step}{0}

\begin{step}\label{s1-birat-connected}
For any $i \in \{0, \ldots, \ell\}$, it holds that  
$$\Nklt \left( X_i, \Delta_{X_i} \right) =\Nklt \left( X_i, D_{X_i} \right) =
\Supp \Delta_{X_i}^{\geq 1} = \Supp D_{X_i}^{= 1}.$$
\end{step}

\begin{proof}[Proof of Step \ref{s1-birat-connected}]
Fix $i \in \{0, \ldots, \ell\}$. 
We obtain 
\[
\Nklt \left( X_i, D_{X_i} \right) 
=\Supp D_{X_i}^{= 1}=\Supp \Delta_{X_i}^{\geq 1} \subset \Nklt \left( X_i, \Delta_{X_i} \right), 
\]
where the first equality holds by (1)' and the second one follows from (2)'. 
Hence, it is sufficient to show that 
$\Nklt \left( X_i, \Delta_{X_i} \right) \subset \Supp D_{X_i}^{= 1}$. 
For sufficiently large $b > 0$, the $\R$-divisor 
\[
A_{X_i} := (\psi _i ^{-1}) _* \Delta + (1 - \epsilon) F_{X_i} + b G_{X_i}
\]
satisfies $\Delta_{X_i} \le A_{X_i}$. 
Therefore, we get 
\[
\Nklt \left( X_i, \Delta_{X_i} \right) \subset \Nklt \left( X_i, A_{X_i} \right). 
\]
Since $A_{X_i}^{\wedge 1} = D_{X_i}$, we have $A_{X_i}^{<1}=D_{X_i}^{<1}$, hence $(X, A_{X_i}^{<1})$ is klt by (1)'. 
Thus,  it follows from Proposition \ref{p-nklt-non-negative} that 
\[
\Nklt \left( X_i, A_{X_i} \right) = \Supp A_{X_i} ^{\ge 1} = \Supp D_{X_i} ^{= 1}. 
\]
Thus, we obtain the desired inclusion $\Nklt \left( X_i, \Delta_{X_i} \right) \subset \Supp D_{X_i}^{= 1}$. 
This completes the proof of Step \ref{s1-birat-connected}. 
\end{proof}

\begin{step}\label{s2-birat-connected}
Let $g:X_i \to X_{i+1}$ be a divisorial contraction appearing in the MMP (\ref{e1-birat-connected}). 
If $\Nklt \left( X_{i+1}, D_{X_{i+1}} \right) \to V$ has connected fibres, 
then so does $\Nklt \left( X_i, D_{X_i} \right) \to V$. 
\end{step}

\begin{proof}[Proof of Step \ref{s2-birat-connected}] 
It follows from Lemma \ref{l-strong-Birkar} 
that $\Nklt \left( X_i, D_{X_i} \right) \to \Nklt \left( X_{i+1}, D_{X_{i+1}} \right)$ has connected fibres. 
This completes the proof of Step \ref{s2-birat-connected}. 
\end{proof}

\begin{step}\label{s3-birat-connected}
Let $h:X_i \dashrightarrow X_{i+1}$ be a flip appearing in the MMP (\ref{e1-birat-connected}).  
If $\Nklt \left( X_{i+1}, D_{X_{i+1}} \right) \to V$ has connected fibres, 
then so does $\Nklt \left( X_i, D_{X_i} \right) \to V$. 
\end{step}

\begin{proof}[Proof of Step \ref{s3-birat-connected}] 
Assume that $\Nklt \left( X_{i+1}, D_{X_{i+1}} \right) \to V$ has connected fibres. 
Let $\varphi_i:X_i \to Y$ be the flipping contraction and 
let $\varphi_{i+1}:X_{i+1} \to Y$ and $\psi_Y:Y \to V$ be the induced morphisms. 
Set $N_Y:=\varphi_i \left( \Nklt \left( X_i, D_{X_i} \right) \right)$ and $N_V:=\psi_Y(N_Y)$. 
It follows from Step \ref{s1-birat-connected} that 
$N_Y=\varphi_{i+1} \left( \Nklt \left( X_{i+1}, D_{X_{i+1}} \right) \right)$. 
By assumption, it holds that the composite morphism 
$$\Nklt \left( X_{i+1}, D_{X_{i+1}} \right) \to N_Y \to N_V$$ 
is a surjective morphism with connected fibres. 
In particular, $N_Y \to N_V$ has connected fibres. 
Since $\Nklt \left( X_i, D_{X_i} \right) \to N_Y$ has connected fibres by Lemma \ref{l-strong-Birkar}, 
their composition
\[
\Nklt \left( X_i, D_{X_i} \right) \to N_Y \to N_V
\]
also has connected fibres. This completes the proof of Step \ref{s3-birat-connected}. 
\end{proof}

\begin{step}\label{s4-birat-connected}
The induced morphism $\Nklt \left( X_{\ell}, D_{\ell} \right) \to V$ has connected fibres. 
\end{step}

\begin{proof}[Proof of Step \ref{s4-birat-connected}] 
We have 
\[
B_{\ell} := \left( \psi _{\ell} ^{-1} \right) _* \Delta  ^{\wedge 1} + 
(1 - \epsilon) F_{X_{\ell}} + G _{X_{\ell}} - \Delta _{X_{\ell}}
\sim_{V, \R} 
K_{X_{\ell}} + D_{X_{\ell}}. 
\]
Then $B_{\ell}$ is nef over $V$. 
The push-forward of $-B_{\ell}$ on $V$, 
which is nothing but the push-forward of 
$\Delta _{X_{\ell}} - \left( \psi _{\ell} ^{-1} \right) _* \Delta ^{\wedge 1}$, is effective. 
Hence, it turns out by the negativity lemma that $-B_{\ell}$ itself is effective. 
Since $\epsilon$ is sufficiently small, it follows that $F_{X_{\ell}} = 0$, that is, 
any $\varphi_{\ell}$-exceptional prime divisor $E$ 
satisfies $a_E(V, \Delta) \leq 0$. 
Since $V$ is $\Q$-factorial, it holds that 
$$\Ex \left( \varphi_{\ell} \right) \subset \Nklt \left( X_{\ell}, \Delta_{X_{\ell}} \right).$$
In particular, 
$\Nklt \left( X_{\ell}, \Delta_{X_{\ell}} \right) \cap \varphi_{\ell} ^{-1} (v) = \varphi_{\ell} ^{-1}(v)$ 
holds and this is connected for any closed point $v$ of $V$. 
This completes the proof of Step \ref{s4-birat-connected}. 
\end{proof}
Step \ref{s2-birat-connected}, Step \ref{s3-birat-connected}, and 
Step \ref{s4-birat-connected} complete the proof of Proposition \ref{p-birat-connected}. 
\end{proof}

\begin{thm}\label{t-birat-connected2}
Let $k$ be a perfect field of characteristic $p>5$. 
Let $f:X \to V$ be a projective birational $k$-morphism 
of normal quasi-projective threefolds over $k$. 
Let $(X, \Delta)$ be a sub-log pair over $k$ such that $-(K_X+\Delta)$ is $f$-nef and $f_*\Delta$ is effective. 
Then the induced morphism $\Nklt(X, \Delta) \to V$ has connected fibres. 
\end{thm}

\begin{proof}
Taking the base change to the algebraic closure of $k$, 
we may assume that $k$ is an algebraically closed field. 
We now reduce the problem to the case when $K_X+\Delta \sim_{\mathbb{R},f} 0$. 
Since $f$ is birational, $-(K_X+\Delta)$ is $f$-nef and $f$-big. 
After replacing $\Delta$, we may assume that $-(K_X+\Delta)$ is $f$-ample. 
Then there exists an effective $\R$-Cartier $\R$-divisor $A$ 
such that $A \sim_{\R, f} -(K_X+\Delta)$ 
and $\Nklt(X, \Delta) = \Nklt(X, \Delta+A)$. 
Thus, we may assume that $K_X+\Delta \sim_{\mathbb{R}, f} 0$. 
In particular, for $\Delta_V:=f_*\Delta$, 
it holds that $(V, \Delta_V)$ is a log pair and $K_X+\Delta=f^*(K_V+\Delta_V)$. 

Let $\varphi:V_1 \to V$ be a dlt modification of $(V, \Delta_V)$ 
such that $\Nklt \left( V_1, \Delta_{V_1} \right)=\varphi^{-1} \left( \Nklt \left( V, \Delta_V \right) \right)$ (Proposition \ref{p-dlt-modif}). 
In particular, $\Nklt \left( V_1, \Delta_{V_1} \right) \to \Nklt(V, \Delta_V)$ 
has connected fibres. 
Let $f_1:X_1 \to V_1$ be a log resolution of $(V_1, \Delta_{V_1})$ 
that factors through $X$. 
By Proposition \ref{p-birat-connected}, 
$\Nklt \left( X_1, \Delta_{X_1} \right) \to \Nklt \left( V_1, \Delta_{V_1} \right)$ 
has connected fibres. 
Thus, the composite morphism 
$$\Nklt \left( X_1, \Delta_{X_1} \right) \to \Nklt \left( V_1, \Delta_{V_1} \right) \to \Nklt(V, \Delta_V)$$
has connected fibres and factors through $\Nklt(X, \Delta)$. 
In particular, also $\Nklt(X, \Delta) \to \Nklt(V, \Delta_V)$ 
has connected fibres. 
\end{proof}

\begin{rem}
When we apply Proposition \ref{p-birat-connected} in the above proof, 
we only use the properties that $V_1$ is $\Q$-factorial 
and $\Nklt \left( V_1, \Delta_{V_1} \right) =\varphi^{-1} \left( \Nklt \left(V, \Delta_V \right) \right)$, 
whilst we do not use the fact that $\left( V_1, \Delta_{V_1}^{\wedge 1} \right)$ is dlt. 
\end{rem}

\subsection{Results on the Witt vector cohomologies}

For the definition of the Witt vector cohomology and its properties, 
we refer to \cite{GNT} and \cite{CR12}. 
Our goal of this subsection is to show Proposition \ref{p-WO-bc} and Proposition \ref{p-rat-curve2}. 
As far as the authors know, 
it is an open problem whether $R^if_*(W\MO_{X, \Q})$ commute with base changes. 
Such a problem occur because inverse limits do not commute with tensor products. 
We start by showing some auxiliary results.

\begin{lem}\label{l-GNT2.22} 
Let $k$ be a perfect field of characteristic $p>0$. 
Let $f:X \to Y$ be a proper surjective $k$-morphism of separated
schemes of finite type over $k$. 
Then the following conditions are equivalent. 
\begin{enumerate}
\item $f$ has connected fibres. 
\item The induced homomorphism $W\MO_{Y, \Q} \to f_*W\MO_{X, \Q}$ is an isomorphism. 
\end{enumerate}
\end{lem}

\begin{proof}
It follows from \cite[Lemma 2.22]{GNT} that (1) implies (2).

It is enough to show that (2) implies (1). 
Assume (2). 
Taking the Stein factorisation, the problem is reduced to the case when 
$f$ is a finite surjective morphism. 
Since the problem is local on $Y$, we may assume that $X=\Spec\,B$ and $Y=\Spec\,A$. 
For the induced ring homomorphism $A \to B$, 
we have that $W(A)_{\Q} \to W(B)_{\Q}$ is an isomorphism. 
Fix a maximal ideal $\mfm$ of $A$. 
We have the following commutative diagram of ring homomorphisms: 
\[
\begin{CD}
W(A)_{\Q} @>\simeq >> W(B)_{\Q}\\
@VVV @VVV\\
W(A/\mfm)_{\Q} @>\psi>> W(B/\mfm B)_{\Q}.\\
\end{CD}
\]
By a diagram chase, 
$\psi:W (A/\mfm)_{\Q} \to W(B/\mfm B)_{\Q}$ 
is surjective. 
On the other hand, $W(A/\mfm )_{\Q}$ is a field, hence 
$\psi$ is an isomorphism. 
In particular, $f^{-1}(\mfm)$ consists of one point $\mfn$. 
By $W(A/\mfm)_{\Q} \simeq W(B/\mfm B)_{\Q}$ 
and $W(B/\mfn)_{\Q} \simeq W(B/\mfm B)_{\Q}$, 
we have $W(A/\mfm)_{\Q} \simeq W(B/\mfn)_{\Q}$. 
Hence, the finite extension $W(A/\mfm) \hookrightarrow W(B/\mfn)$ of discrete valuation rings  
is also an isomorphism. 
Taking modulo $p$ reduction, we have that $A/\mfm \to B/\mfn$ is an isomorphism. 
Thus (1) holds. 
\end{proof}

We often use the following exact sequences, 
which we call the Mayer--Vietoris exact sequences. 
\begin{lem}\label{l-MV} 
Let $k$ be a perfect field of characteristic $p>0$. 
Let $V$ be a scheme of finite type over $k$. 
Let $X, X_1$, and $X_2$ be closed subschemes of $V$ 
such that the set-theoretic equation $X=X_1 \cup X_2$ holds. 
Let $X_1 \cap X_2$ be the scheme-theoretic intersection. 
Let $I_X, I_{X_1}, I_{X_2},$ and $I_{X_1 \cap X_2}$ be the corresponding coherent ideal sheaves on $V$. 
Then there exist the exact sequences  
\begin{enumerate}
\item 
$0 \to W\MO_{X, \Q} \to W\MO_{X_1, \Q} \oplus W\MO_{X_2, \Q} \to W\MO_{X_1 \cap X_2, \Q} \to 0$, 
and 
\item 
$0 \to WI_{X, \Q} \to WI_{X_1, \Q} \oplus WI_{X_2, \Q} \to WI_{X_1 \cap X_2, \Q} \to 0$.
\end{enumerate}
\end{lem}

\begin{proof}
By using Remark \ref{r-ideal-radical} and the fact that the functor $(-)_{\Q}$ is exact, 
we obtain the exact sequence (1) by the same argument as in \cite[Proposition 2.2]{BBE07}.
The exact sequence (2) is obtained by (1) and the snake lemma. 
\end{proof}

\begin{lem}\label{l-H0-bc}
Let $k \subset k'$ be an extension of 
perfect fields of characteristic $p>0$. 
Let $X$ be a proper scheme over $k$ and we set $X':=X \times_k k'$. 
Then the induced $W\left( k' \right)_{\Q}$-linear map 
$$H^0 \left( X, W\MO_{X, \Q} \right) \otimes_{W(k)_{\Q}} W(k')_{\Q} \to 
H^0 \left( X', W\MO_{X', \Q} \right)$$
is bijective. 
\end{lem}

\begin{proof}
Taking the Stein factorisation of the structure morphism 
$X \to \Spec\,k$, 
we may assume that $X$ is of dimension zero. 
Replacing $X$ by a connected component, 
we may assume that $X=\Spec\,L$, where $L$ is a finite extension of $k$. 
Then the assertion is clear. 
\end{proof}

To prove Proposition \ref{p-rat-curve2}, we first show the following weaker statement. 

\begin{lem}\label{l-rat-curve}
Let $k$ be a perfect field of characteristic $p>0$ and 
let $X$ be a one-dimensional smooth projective scheme over $k$. 
Then the following are equivalent. 
\begin{enumerate}
\item $H^1 \left( X, \MO_X \right) = 0$. 
\item $H^1 \left( X, W_n\MO_X \right) =0$ for some $n \in \Z_{>0}$.  
\item $H^1 \left( X, W_n\MO_X \right) =0$ for any $n \in \Z_{>0}$. 
\item $H^1 \left( X, W\MO_X \right)=0$. 
\item $H^1 \left( X, W\MO_{X, \Q} \right)=0$. 
\end{enumerate}
\end{lem}

\begin{proof}
Clearly we may assume that $X$ is connected. 
By the exact sequence 
\[
0 \to W_n\MO_X \xrightarrow{V} W_{n+1}\MO_X \to \MO_X \to 0, 
\]
and the fact that $X$ is one-dimensional, 
it holds that (1), (2) and (3) are equivalent. 
By \cite[Lemma 2.19]{GNT}, (3) implies (4). 
Moreover, by the exact sequence 
$$0 \to W\MO_X \xrightarrow{V} W\MO_X \to \MO_X \to 0$$
and the fact that $X$ is one-dimensional, 
(4) implies (1). 

The equivalence between (4) and (5) follows from the fact that $H^1 \left ( X, W\MO_X \right )$ is a free $W(k)$-module (\cite[Ch. II, Proposition 2.19]{Ill79}). 
\end{proof}

\begin{lem}\label{l-tree-bc}
Let $k \subset k'$ be an extension of 
perfect fields of characteristic $p>0$. 
Let $X$ be a proper one-dimensional scheme over $k$. 
Then the following are equivalent. 
\begin{enumerate}
\item $H^1 \left( X, W\MO_{X, \Q} \right)=0$. 
\item $H^1 \left( X \times_k k', W\MO_{X \times_k k', \Q} \right)=0$. 
\end{enumerate}
\end{lem}

\begin{proof}
For simplicity, we denote $K =W(k)_{\Q}$, $K' = W \left( k' \right)_{\Q}$, and $Y' = Y \times _k k'$ for a $k$-scheme $Y$. 
We may assume that $X$ is reduced. 
Let 
$$X^N \to X$$
be the normalisation of $X$. 
Thanks to Lemma~\ref{l-rat-curve}, 
if one of (1) and (2) holds, then it follows that 
\[
H^1 \left( X^N, W\MO_{X, \Q} \right ) =H^1 \left( X'^N, W\MO_{X'^N, \Q} \right)=0.
\]
For the conductor subschemes $C$ and $D$ of $X$ and $X^N$ respectively, 
we have a commutative diagram with exact horizontal sequences: 
$$\begin{CD}
\begin{matrix}
H^0 ( W\MO_{X^N, \Q})_{K'} \\
\ \ \oplus \ 
H^0 ( W\MO_{C, \Q})_{K'}
\end{matrix}
 @>>> H^0(W\MO_{D, \Q})_{K'} @>>> 
H^1(W\MO_{X, \Q})_{K'} @>>> 0\\
@VV\alpha V @VV\beta V @VV\gamma V\\
\begin{matrix}
H^0(W\MO_{X'^N, \Q}) \\ 
\ \ \oplus \ 
H^0(W\MO_{C', \Q})
\end{matrix}
@>>> H^0(W\MO_{D', \Q}) @>>> H^1(W\MO_{X', \Q}) @>>> 0,\\
\end{CD}$$
where $(-)_{K'}$ denotes the tensor product $(-) \otimes_K K'$. 
As both $\alpha$ and $\beta$ are isomorphisms by Lemma \ref{l-H0-bc}, 
so is $\gamma$ by the 5-lemma, as desired.  
\end{proof}

\begin{prop}\label{p-WO-bc}
Let $k \subset k'$ be an extension of 
perfect fields of characteristic $p>0$. 
Let $X$ be a normal surface over $k$. 
Then the following are equivalent. 
\begin{enumerate}
\item $X$ has $W\MO$-rational singularities.  
\item $X \times_k k'$ has $W\MO$-rational singularities.
\end{enumerate}
\end{prop}

\begin{proof}
We may assume that $Q$ is a unique non-regular point of $X$. 
Let $f:Y \to X$ be a resolution of singularities such that $f\left( \Ex \left( f \right) \right)=Q$. 
For $E:=\Ex(f)$, we have the exact sequence: 
\[
0 \to WI_{E, \Q} \to W\MO_{Y, \Q} \to W\MO_{E, \Q} \to 0.
\]
Thanks to the vanishing $R^if_* \left( WI_{E, \Q} \right)=0$ for $i>0$ \cite[Theorem 2.4]{BBE07}, 
it holds that 
\[R^if_* \left( W\MO_{Y, \Q} \right) \simeq H^i \left( E, W\MO_{E, \Q} \right).\]
Therefore, it follows from Lemma~\ref{l-tree-bc} that 
(1) and (2) are equivalent. 
\end{proof}

\begin{prop}\label{p-rat-curve2}
Let $k$ be a perfect field of characteristic $p>0$ and 
let $X$ be a reduced projective scheme over $k$ such that 
\begin{enumerate}
\item[(a)] any irreducible component of $X$ is one-dimensional, and 
\item[(b)] any non-regular point $x$ of $X$ is an ordinary double point. 
\end{enumerate}
Then the following are equivalent. 
\begin{enumerate}
\item $H^1 \left( X, \MO_X \right)=0.$  
\item $H^1 \left( X, W_n\MO_X \right )=0$ for some $n \in \Z_{>0}$.  
\item $H^1(X, W_n\MO_X)=0$ for any $n \in \Z_{>0}$. 
\item $H^1(X, W\MO_X)=0$. 
\item $H^1(X, W\MO_{X, \Q})=0$. 
\item Any connected component of $X \times_k \overline{k}$ is a tree of smooth rational curves. 
\end{enumerate}
\end{prop}

\begin{proof}
We may assume that $X$ is connected. 
Moreover, replacing $k$ by $k'$ for the Stein factorisation 
$X \to \Spec\, k' \to \Spec\,k$, we may assume that $X$ is geometrically connected. 

We now show that the assertions (1), (2), (3) and (4) are equivalent. 
By the exact sequence 
$$0 \to W_n\MO_X \xrightarrow{V} W_{n+1}\MO_X \to \MO_X \to 0$$
and the fact that $X$ is one-dimensional, 
it holds that (1), (2) and (3) are equivalent. 
We have that (3) implies (4) by \cite[Lemma 2.19]{GNT}. 
Moreover, by the exact sequence 
$$0 \to W\MO_X \xrightarrow{V} W\MO_X \to \MO_X \to 0$$
and the fact that $X$ is one-dimensional, 
(4) implies (1). 
Thus (1), (2), (3) and (4) are equivalent.

Thanks to \cite[Ch II, Lemma 7.5]{Kol96}, it holds that (6) implies (1). 
Further, (4) clearly implies (5). 
Thus it suffices to show that (5) implies (6). 
Lemma \ref{l-tree-bc} allows us to replace $X \to \Spec\,k$ by the base change 
$X \times_k \overline k \to \Spec\,\overline k$.  
Then it follows from 
\cite[the second last paragraph of Section 4.6]{CR12}
that (5) implies (6), as desired. 
\end{proof}

\subsection{Geometric rationality of del Pezzo surfaces over imperfect fields}

In this subsection, 
we prove Proposition \ref{p-klt-dP}. 
To this end, we start with the following lemma.

\begin{lem}\label{l-rationality}
Let $k$ be a separably closed field of characteristic $p>0$ 
which is not algebraic over a finite field. 
Let $X$ be a projective normal $\Q$-factorial surface over $k$ 
with $k=H^0(X, \MO_X)$. 
If there is an $\R$-divisor $\Delta$ 
such that $0 \leq \Delta <1$ and $-(K_X+\Delta)$ is nef and big,  
then $\left( X \times_k \overline{k} \right )_{\red}$ is a rational surface. 
\end{lem}

\begin{proof}
Replacing $\Delta$, we may assume that $-(K_X+\Delta)$ is ample. 
If $X \to X'$ is a birational $k$-morphism of projective normal varieties 
with $k=H^0(X, \MO_X)=H^0 \left( X', \MO_{X'} \right)$, 
then also $\left( X \times_k \overline k \right) _{\red} \to \left( X' \times_k \overline k \right)_{\red}$ 
is birational. 
Thus we may replace $(X, \Delta)$ by the end result of 
a $(K_X+\Delta)$-MMP (\cite[Theorem 1.1]{Tanc}). 
Hence we may assume that one of the following condition holds. 
\begin{enumerate}
\item[(a)] $\rho(X)=1$. 
\item[(b)] There is a $(K_X+\Delta)$-Mori fibre space 
$\pi_1:X \to B_1$ onto a regular projective curve $B_1$ with $(\pi_1)_*\MO_X=\MO_{B_1}$. 
\end{enumerate}

In what follows, we denote by $Y$ the normalisation of $\left( X \times_k \overline k \right) _{\red}$, and denote by $f:Y \to X$ the composite morphism. 
By applying \cite[Theorem 1.1]{Tana} to the regular locus of $X$, we can write 
$$K_Y+D=f^*K_X$$
for some effective $\Z$-divisor $D$.

Suppose (a). Then $Y$ is a projective normal $\Q$-factorial surface such that 
$\rho(Y)=1$ (\cite[Proposition 2.4 (2)]{Tana}). 
Since $-K_Y$ is ample, $Y$ is a ruled surface. 
Assume that $Y$ is not rational and let us derive a contradiction. 
Let $\mu:Z \to Y$ be the minimal resolution of $Y$. 
Since $Z$ is an irrational ruled surface, 
there is a projective morphism $\pi:Z \to B$ 
onto a smooth projective irrational curve 
whose general fibres are $\mathbb P^1$. 
Since $\overline k \neq \overline{\mathbb F}_p$, 
it follows from \cite[Theorem 3.20]{Tan14} that $\pi$ factors through $\mu$: 
$$\pi:Z \xrightarrow{\mu} Y \to B.$$
This is a contradiction to $\rho(Y)=1$. 
Thus we are done for the case (a).

Suppose (b). 
Since $-(K_X+\Delta)$ is ample, 
there exists an extremal ray $R$ of $\overline{{\rm NE}}(X)$ 
not corresponding to $\pi_1$. 
By \cite[Theorem 4.4]{Tanc}, 
the extremal ray $R$ induces either a birational morphism or 
another $(K_X+\Delta)$-Mori fibre space $X \to B_2$ onto a curve $B_2$. 
If the former case occurs, then the problem is reduced to the case (a). 
Therefore, we may assume that there exist two 
Mori fibre space structures $\pi_1:X \to B_1$ and $\pi_2:X \to B_2$ 
onto curves $B_1$ and $B_2$. 
In particular, any fibre of $\pi_i$ dominates $B_{3-i}$. 
Let $\pi'_i:Y \to B'_i$ be the Stein factorisation of the composite morphism: 
$$Y \to X \times_k \overline k \xrightarrow{\pi_i \times_k \overline k} 
B_i \times_k \overline k.$$
Then any fibre of $\pi'_i$ dominates $B'_{3-i}$. 
Since $-K_Y$ is big, a general fibre of each $\pi'_i$ 
is isomorphic to $\mathbb P^1$. 
In particular, $B'_1 \simeq \mathbb P^1$ and $Y$ is rational. 
\end{proof}

\begin{rem}
The statement of Lemma \ref{l-rationality} does not hold 
if we drop the assumption on the base field $k$. 
Indeed if $k=\overline{\mathbb F}_p$, then 
any normal surface is $\Q$-factorial (e.g.\ see \cite[Theorem 4.5]{Tan14}). 
Thus the cone $X$ over an elliptic curve over $\overline{\mathbb F}_p$ 
is $\Q$-factorial and $-K_X$ is ample. 
\end{rem}

\begin{prop}\label{p-klt-dP}
Let $(X, \Delta)$ be a projective two-dimensional klt pair over a field $k$ 
of characteristic $p>0$ such that $-(K_X+\Delta)$ is nef and big. 
Assume that $k=H^0(X, \MO_X)$.  
Then, $\left( X \times_k \overline{k} \right)_{\red}$ is a rational surface. 
In particular $X$ is rationally connected over $k$. 
\end{prop}

\begin{proof} 
We may assume that $k$ is separably closed. 
Since the assertion is well-known if $k$ is an algebraically closed field 
(cf.\ \cite[Fact 3.4 and Theorem 3.5]{Tan15}), 
the problem is reduced to the case when $k$ is an imperfect field. 
In particular, $k$ is not algebraic over a finite field. 
As $X$ is $\Q$-factorial \cite[Corollary 4.11]{Tanc}, the assertion follows from Lemma~\ref{l-rationality}. 
\end{proof}

\section{$W\MO$-vanishing for log Fano contractions}\label{s-WO-MFS}

In this section, we prove a vanishing theorem for log Fano contractions (Theorem \ref{t-rel-van}). 
We shall divide the proof into the cases depending on the dimension of the base scheme $Z$. 
The cases $\dim Z=1$ and $\dim Z=2$ 
are treated in Subsection \ref{ss-dP-fib} and Subsection \ref{ss-conic-bdl} respectively. 
The remaining cases $\dim Z=0$ and $\dim Z=3$ 
have been already settled in \cite{GNT} (cf.\ the proof of Theorem \ref{t-rel-van}).

Before starting the case study, 
we summarise some results used repeatedly in the proof of Theorem \ref{t-rel-van}.

\begin{thm}\label{t-gnt}
Let $k$ be a perfect field of characteristic $p>5$. Then the following hold. 
\begin{enumerate}
\item[(1)] If $(X, \Delta)$ is a three-dimensional klt pair over $k$, then $X$ has $W \mathcal{O}$-rational singularities. 
\item[(2)] If $X$ is a three-dimensional projective variety of Fano type over $k$, 
then $H^i \left( X, W \mathcal{O}_{X, \mathbb{Q}} \right) = 0$ for $i > 0$. 
\end{enumerate}
\end{thm}
\begin{proof}
When $\Delta$ is $\mathbb{Q}$-divisor, both (1) and (2) follow from \cite[Theorem 1.4]{GNT} and \cite[Theorem 1.3]{GNT}, respectively. 
Thanks to \cite[Lemma 4.6.1]{Fuj17}, the general case is reduced to this case. 
\end{proof}

\begin{thm}\label{thm:CR481}
Let $f: X \to Y$ be a projective morphism between integral schemes with $W\mathcal{O}$-rational singularities. 
Suppose that $Y$ is normal and that the generic fibre $X_{K(Y)}$ of $f$ is smooth and rationally chain connected. 
Then $R^if_*W\mathcal{O}_{X, \mathbb{Q}} = 0$ holds for $i > 0$. 
\end{thm}

\begin{proof}
This is a special case of \cite[Theorem 4.8.1]{CR12}. 
\end{proof}

\begin{lem}\label{l-geom-normalise}
Let $k$ be a field of characteristic $p>0$. 
Let $f:X \to Y$ be a projective $k$-morphism 
of normal $k$-varieties such that $f_*\MO_X=\MO_Y$. 
Then there exists a commutative diagram 
$$\begin{CD}
X' @>\alpha >> X\\
@VVf'V @VVf V\\
Y' @>\beta >> Y
\end{CD}$$
of projective $k$-morphisms of normal $k$-varieties 
that satisfies the following properties. 
\begin{enumerate}
\item Both $\alpha$ and $\beta$ are finite universal homeomorphisms. 
\item $f'_*\MO_{X'}=\MO_{Y'}$. 
\item The generic fibre $X'_{K(Y')}$ of $f'$ is geometrically normal over $K(Y')$. 
\item The induced morphism 
$X'_{K(Y')} \to \left( X \times_Y K \left( Y' \right) \right)_{\red}$ 
is a finite birational morphism. 
In particular, this morphism coincides with the normalisation of $\left( X \times_Y K \left( Y' \right) \right)_{\red}$. 
\end{enumerate}
\end{lem}

\begin{proof}
We set $K:=K(Y)$. 
Let $\nu_0:X'_0 \to \left( X \times_Y K^{1/p^{\infty}} \right)_{\red}$ be 
the normalisation of $\left( X \times_Y K^{1/p^{\infty}} \right)_{\red}$. 
Since $\nu_0$ is a finite universal homeomorphism by \cite[Lemma 2.2]{Tana}, 
we have that $X'_0$ is geometrically connected and projective over a perfect field $K^{1/p^{\infty}}$, 
hence $K^{1/p^{\infty}}=H^0 \bigl( X'_0, \MO_{X'_0} \bigr)$. 
There exist an intermediate field $L$ between $K$ and $K^{1/p^{\infty}}$ 
satisfying $[L:K]<\infty$ and 
a projective normal $L$-variety $X'_1$ 
such that $X'_1 \times_{L} K^{1/p^{\infty}}=X'_0$ 
with the following commutative diagram, where $\nu_1$ is birational:  
\[
\begin{CD}
X' _0 @>>> X_1' @. @. \\
@VV\nu _0V @VV \nu _1V @. @. \\
(X \times_Y K^{1/p^{\infty}})_{\red} @>>> (X \times_Y L)_{\red} @>>> X \times_Y K @>>> X\\
@VVV @VVV @VVV @VVV \\
\mathrm{Spec}\ K^{1/p^{\infty}} @>>> \mathrm{Spec}\ L @>>> \mathrm{Spec}\ K @>>> Y
\end{CD}
\]
\vspace{2mm}

\noindent
In particular, it follows that $L=H^0 \bigl( X'_1, \MO_{X'_1} \bigr)$ and 
$\nu _1$ is a finite universal homeomorphism. 
Since $\nu _1$ is a finite birational morphism and $X_1 '$ is normal, the $\nu _1$ is nothing but the normalisation of $(X \times_Y L)_{\red}$. 
Furthermore, $X_1 '$ is geometrically normal, since $X_0 '$ is normal and $K^{1/p^{\infty}}$ is perfect. 

Let $X'$ (resp.\ $Y'$) be the normalisation of $X$ (resp.\ $Y$) 
in $K(X'_1)$ (resp.\ $L$). 
Then we get the commutative diagram as in the statement, 
and the properties (1), (3) and (4) follow from the construction. 

Let us show (2). 
Since $\MO_{Y'} \to f'_*\MO_{X'}$ is an isomorphism 
on some non-empty open subset of $Y'$, 
it holds that $Y'' \to Y'$ is a finite birational morphism for 
the Stein factorisation 
$$f':X' \to Y'' \to Y'$$
of $f'$. 
As $Y'$ is normal, we have that $Y'' \to Y'$ is an isomorphism, 
hence (2) holds. 
\end{proof}

\subsection{Del Pezzo fibrations}\label{ss-dP-fib}

In this subsection, we establish the $W\MO$-vanishing 
for del Pezzo fibrations (Proposition \ref{p-dP-fib}). 
A key result is the following.

\begin{lem}\label{l-dP-CR}
Let $k$ be a perfect field of characteristic $p>0$. 
Let $f:X \to Y$ be a projective $k$-morphism such that 
\begin{enumerate}
\item $X$ is a normal threefold over $k$ that has $W\MO$-rational singularities, 
\item $Y$ is a smooth $k$-curve, 
\item $f_*\MO_X=\MO_Y$, and 
\item the geometric generic fibre 
$X_{\overline{K(Y)}}$ of $f$ is a normal rational surface. 
\end{enumerate}
Then $R^if_* \left( W\MO_{X, \Q} \right)=0$ for $i>0$. 
\end{lem}

\begin{proof}
We divide the proof into two steps. 
\setcounter{step}{0}
\begin{step}\label{step1-dP-CR}
The assertion of Lemma \ref{l-dP-CR} holds if there exists 
a projective birational $K(Y)$-morphism 
$$g_0:Z_0 \to X_{K(Y)}$$
from a smooth projective $K(Y)$-surface $Z_0$. 
\end{step}

\begin{proof}[Proof of Step \ref{step1-dP-CR}]
Killing the denominators of all the elements of $K(Y)$ defining $g_0$, 
we can find a non-empty open subset $Y'$ of $Y$ and 
morphisms 
$$h':Z' \xrightarrow{g'} X':=f^{-1}(Y') \xrightarrow{f|_{f^{-1}(Y')}} Y'$$
whose base changes by $(-) \times_{Y'} \Spec\,K(Y)$ are the same as 
\[
Z_0 \xrightarrow{g_0} X_{K(Y)} \to \Spec\,K(Y). 
\]
Furthermore, we may assume that 
\begin{itemize}
\item $Z'$ is an integral scheme, 
\item $g'$ is a projective birational morphism, and 
\item the composite morphism $h'$ is smooth. 
\end{itemize}
In particular, $Z'$ is a smooth threefold over $k$. 
Let $g:Z \to X$ be a smooth projective compactification of the induced morphism 
$$Z' \xrightarrow{g_1} X'=f^{-1}(Y') \hookrightarrow X,$$ 
i.e.\ there are morphisms 
$$g':Z' \xrightarrow{j} Z \xrightarrow{g} X$$ 
such that $j$ is an open immersion, $g$ is projective, and 
$Z$ is an integral scheme smooth over $k$. 
In particular $Z$ is a smooth threefold over $k$ which is projective over $X$, hence over $Y$. 
We get the composite morphism 
$$h:Z \xrightarrow{g} X \xrightarrow{f} Y$$
whose geometric generic fibre $Z \times_Y  \overline{K(Y)}$ satisfies the following isomorphisms: 
$$Z \times_Y  \overline{K(Y)} \simeq Z' \times_{Y'} \overline{K(Y)} \simeq Z_0 \times_{K(Y)} \overline{K(Y)}.$$
In particular, $Z \times_Y  \overline{K(Y)}$ is 
a smooth projective rational surface over $\overline{K(Y)}$. 

Therefore we have that 
$$Rf_*(W\MO_{X, \Q}) \simeq Rf_*Rg_*(W\MO_{Z, \Q}) 
\simeq Rh_*(W\MO_{Z, \Q})  
\simeq W\MO_{Y, \Q},$$
where the first isomorphism 
follows from the assumption (1) and the third follows from Theorem \ref{thm:CR481}. 
This completes the proof of Step \ref{step1-dP-CR}. 
\end{proof}

\begin{step}\label{step2-dP-CR}
The assertion of Lemma \ref{l-dP-CR} holds without any additional assumptions. 
\end{step}

\begin{proof}[Proof of Step \ref{step2-dP-CR}]
Let $K(Y)^{1/p^{\infty}}$ be the purely inseparable closure of $K(Y)$ in the algebraic closure $\overline{K(Y)}$ of $K(Y)$. 
We fix a projective birational $K(Y)^{1/p^{\infty}}$-morphism 
$$g_1:Z_1 \to X \times_Y  K(Y)^{1/p^{\infty}}$$
from a regular $K(Y)^{1/p^{\infty}}$-surface $Z_1$. 
Note that $Z_1$ is smooth over $K(Y)^{1/p^{\infty}}$, since $K(Y)^{1/p^{\infty}}$ is a perfect field. 
Then there exist a finite purely inseparable extension $L$ of $K(Y)$ 
and a projective normal $L$-surface $Z_2$ with the following projective birational $L_0$-morphism 
\[
g_2:Z_2 \to X \times_Y  L, 
\]
whose base change by $(-) \times_{L} K(Y)^{1/p^{\infty}}$ is isomorphic to $g_1$. 
In particular, $Z_2$ is a smooth projective surface over $L$. 

Let $Y_2$ be the normalisation of $Y$ in $L$ and 
let $X_2$ be the normalisation of $(X \times_Y Y_2)_{\red}$. 
We get a commutative diagram of normal $k$-varieties: 
$$\begin{CD}
X @<\alpha<< X_2 \\
@VVfV @VVf_2V\\
Y @<\beta<< Y_2.
\end{CD}$$

\begin{claim}\label{c-dP-CR}
The following hold. 
\begin{enumerate}
\item[(a)] There exists a non-empty open subset $Y_3$ of $Y_2$ such that 
the induced morphism 
$$X_3:=f_2^{-1}(Y_3) \to X \times_Y Y_3$$
is an isomorphism. 
\item[(b)] The generic fibre of $f_2:X_2 \to Y_2$ is isomorphic to 
$X \times_Y  L$. 
\item[(c)] $\alpha$ is a finite universal homeomorphism. 
\item[(d)] $\beta$ is a finite universal homeomorphism. 
\item[(e)] $(f_2)_*\MO_{X_2}=\MO_{Y_2}$. 
\end{enumerate}
\end{claim}
\begin{proof}[Proof of Claim \ref{c-dP-CR}]
Note that the generic fibre of $X \times_Y Y_2 \to Y_2$ is normal by the geometric normality of $X_{K(Y)}$ (the assumption (4)). 
Therefore (a) holds since $X_2 \to (X \times_Y Y_2)_{\red}$ is the normalisation. 
It is clear that (b) follows from (a). 
As $K(Y) \subset L$ is a purely inseparable extension, 
the assertion (d) holds. 

Let us show (c). 
It follows from the construction that $\alpha:X_2 \to X$ is a finite surjective morphism of normal schemes. 
In particular $X_2$ coincides with the normalisation of $X$ 
in $K(X_2)$. 
Therefore, it suffices to show that the field extension $K(X) \subset K(X_2)$ 
is purely inseparable, which in turn follows from the following equation 
\[
K(X_2)=K(X_2 \times _{Y_2} K(Y_2))=K(X \times_Y L)
\]
where the second equality follows from (b). 
Thus (c) holds. 

Let us show (e). 
Let $f_2:X_2 \to Y'_2 \to Y_2$ be the Stein factorisation of $f_2$. 
By (a), 
there exists a non-empty open subset $Y_4$ of $Y_2$ 
such that the induced homomorphism 
$$\MO_{Y_2}|_{Y_4} \to (f_2)_*\MO_{X_2}|_{Y_4}$$
is an isomorphism. 
In particular $Y'_2 \to Y_2$ is a finite birational morphism 
of integral $k$-varieties. 
As $Y_2$ is normal, it holds that $Y'_2 \to Y_2$ is an isomorphism, 
hence we obtain (e).
This completes the proof of Claim \ref{c-dP-CR}. 
\end{proof}

Let us go back to the proof of Step \ref{step2-dP-CR}. 
Thanks to (c), (d), (e) and (b) of Claim \ref{c-dP-CR}, 
also $f_2$ satisfies the same properties (1), (2), (3) and (4) hold for $f_2$ respectively. 
By Step \ref{step1-dP-CR}, 
it holds that $R^i(f_2)_*(W\MO_{X_2, \Q})=0$ for $i>0$. 
Thanks to (c) and (d) of Claim \ref{c-dP-CR}, 
we have that $R^if_*(W\MO_{X, \Q})=0$ for $i>0$. This completes the proof of Step \ref{step2-dP-CR}. 
\end{proof}
Step \ref{step2-dP-CR} completes the proof of Lemma \ref{l-dP-CR}. 
\end{proof}

\begin{prop}\label{p-dP-fib}
Let $k$ be a perfect field of characteristic $p >5$. 
Let $(X, \Delta)$ be a three-dimensional klt pair over $k$ and 
let $f:X \to Y$ be a projective $k$-morphism to a smooth $k$-curve $Y$ 
such that $f_*\MO_X=\MO_Y$. 
If $-(K_X+\Delta)$ is $f$-nef and $f$-big, 
then $R^if_*(W\MO_{X, \Q})=0$ for $i>0$. 
\end{prop}

\begin{proof}
Applying Lemma \ref{l-geom-normalise} to $f$, we obtain a commutative diagram 
$$\begin{CD}
X' @>\alpha >> X\\
@VVf'V @VVf V\\
Y' @>\beta >> Y
\end{CD}$$
that satisfies the properties listed in Lemma \ref{l-geom-normalise}. 
Since $X$ has $W\MO$-rational singularities and $\alpha$ is 
a finite universal homeomorphism, 
it holds that $X'$ has $W\MO$-rational singularities. 
Furthermore, the geometric generic fibre $X_{\overline{K(Y')}}$ of $f'$ is normal 
(Lemma \ref{l-geom-normalise} (3)) and hence it is a normal rational surface 
by Proposition~\ref{p-klt-dP}. 
Therefore, it follows from Lemma~\ref{l-dP-CR} that 
$Rf'_*(W\MO_{X', \Q})=W\MO_{Y', \Q}$. 
Thus we get 
\[
Rf_* \left( W\MO_{X, \Q} \right) \simeq Rf_* R\alpha_* \left( W\MO_{X', \Q} \right)
\simeq R\beta_* \left( W\MO_{Y', \Q} \right) \simeq W\MO_{Y, \Q},
\]
where the first and last isomorphisms hold because 
$\alpha$ and $\beta$ are finite universal homeomorphisms 
(Lemma \ref{l-geom-normalise} (1)), and 
the second isomorphism follows from $Rf'_* \left( W\MO_{X', \Q} \right)=W\MO_{Y', \mathbb{Q}}$. 
\end{proof}

\subsection{Conic bundles}\label{ss-conic-bdl}

In this subsection, we prove the $W\MO$-vanishing for conic bundles 
(Proposition \ref{p-conic-van}). 
To this end, we show that their base schemes have $W\MO$-rational singularities (Theorem \ref{t-conic-base}). 
Let us start by recalling the following basic fact.

\begin{lem}\label{l-WO-surf-go-up}
Let $k$ be a perfect field of characteristic $p>0$. 
Let $f:X \to Y$ be a proper birational $k$-morphism of normal $k$-surfaces. 
If $Y$ has $W\MO$-rational singularities, then so does $X$. 
\end{lem}

\begin{proof}
Fix a resolution of singularities of $X$: $\varphi:V \to X$. 
We have the following exact sequence induced by 
the corresponding Grothendieck spectral sequence: 
$$0 \to R^1f_* \left( W\MO_{X, \Q} \right) \to R^1 \left( f \circ \varphi \right ) _* \left( W\MO_{V, \Q} \right) \to f_*R^1\varphi_* \left( W\MO_{V, \Q} \right) $$
$$\to R^2f_* \left( W\MO_{X, \Q} \right).$$
We obtain $R^1 \left( f \circ \varphi \right)_* \left( W\MO_{V, \Q} \right) = 0$ 
since $Y$ has $W\MO$-rational singularities. 
Moreover, we have that $R^2f_* \left( W\MO_{X, \Q} \right) = 0$, 
as the fibres of $f$ are at most one-dimensional (cf.\ \cite[Lemma 2.20]{GNT}). 
Therefore, it holds that $f_*R^1\varphi_*(W\MO_{V, \Q}) =0$. 
Thanks to the fact that $\Supp\,R^1\varphi_*(W\MO_{V, \Q})$ is zero-dimensional, we get $R^1\varphi_*(W\MO_{V, \Q})=0$
\end{proof}

\begin{thm}\label{t-conic-base}
Let $k$ be a perfect field of characteristic $p>5$. 
Let $f:X \to Y$ be a projective $k$-morphism of normal $k$-varieties 
with $f_*\MO_X=\MO_Y$ 
which satisfies the following properties. 
\begin{enumerate}
\item{$\dim X=3$ and $\dim Y=2$.}
\item{There exists an effective $\R$-divisor $\Delta$ on $X$ such that 
$(X, \Delta)$ is klt and $-(K_X+\Delta)$ is $f$-nef and $f$-big.}
\end{enumerate}
Then $Y$ has $W\MO$-rational singularities. 
\end{thm}

\begin{proof}
Replacing $\Delta$, we may assume that $-(K_X+\Delta)$ is $f$-ample. 

\setcounter{step}{0}
\begin{step}\label{step-fp-bar}
The assertion of Theorem \ref{t-conic-base} holds 
if $k=\overline{\F}_p$. 
\end{step}

\begin{proof}[Proof of Step~\ref{step-fp-bar}]
In this proof, $X \left( \mathbb F_{p^e} \right)$ denotes the number of  $\mathbb F_{p^e}$-rational points 
on a model $X_0$ of $X$ over $\mathbb F_{p^e}$, 
i.e.\ $X_0$ is a projective $\F_{p^e}$-scheme such that $X_0 \times_{\F_{p^e}} k \simeq X$. 
We can define $X \left( \mathbb F_{p^e} \right) $ if $e$ is a sufficiently divisible positive integer and we fix a model $X_0$ 
(this number possibly depends on the choice of a model $X_0$).  

\medskip 

We may assume that $Y$ has a unique singular point $y$. 
Let $g:Y' \to Y$ be a log resolution such that $g \left( \Ex \left( g \right) \right)=\{y\}$. Set $C:= \Ex(g) = g^{-1}(y)$. 
Let $\varphi: W \to X$ be a log resolution of $(X, \Delta)$ that admits a morphism to $Y'$. 
Then $f ^{-1} (y)$ is rationally chain connected (\cite[Theorem 4.1]{GNT}). 
Hence by \cite[Theorem 4.8]{GNT},  also $(f \circ \varphi)^{-1} (y)$ is rationally chain connected. 
Therefore its image on $Y'$, which is nothing but $C$, is a union of rational curves. 
In order to prove that $Y$ has $W\MO$-rational singularities, 
it is suffices to show that $C$ forms a tree by Proposition \ref{p-rat-curve2} and \cite[Corollary 4.6.4]{CR12}.
Let $s$ be the number of the vertices and let $t$ be the number of the edges of the dual graph of $C$. 
Note that since $C$ is connected and simple normal clossing, the condition that $C$ forms a tree is equivalent to the condition that $s=t+1$. 
Then, for a sufficiently divisible $e$, 
\[
C (\mathbb{F}_{p^e}) = s (p^e + 1) - t
\]
holds, because we may assume that each component of $C$ and their intersection are defined over $\mathbb{F}_{p^e}$. 
Hence, the condition $s=t+1$ is equivalent to the condition that 
\[
C (\mathbb{F}_{p^e}) \equiv 1 \mod p^e
\]
for sufficiently divisible $e$. 
Therefore, it suffices to show that 
\[
Y' \left( \mathbb F_{p^e} \right) \equiv Y \left( \mathbb F_{p^e} \right) \mod p^e
\]
for 
any sufficiently divisible positive integer $e$.

Let $E$ be the sum of all the $\varphi$-exceptional prime divisors. 
We run a $\left( K_W+\varphi_*^{-1}\Delta+E \right)$-MMP over $Y'$ that terminates. 
Since $K_W+\varphi_*^{-1}\Delta+E$ is generically anti-ample over $Y'$, 
we end with a Mori fibre space $X_1 \to Y_1$ over $Y'$. 
Note that the induced morphism $Y_1 \to Y'$ is birational, 
hence $Y_1$ has $W\MO$-rational singularities by Lemma~\ref{l-WO-surf-go-up}. 

Take an arbitrary divisible positive integer $e$. 
We obtain 
$$Y_1 \left( \mathbb F_{p^e} \right) \equiv Y' \left( \mathbb F_{p^e} \right) \mod p^e$$ 
since $Y_1$ and $Y'$ are birational and have $W\MO$-rational singularities (\cite[Corollary 4.4.16]{CR12}). 
Furthermore, it follows from \cite[Theorem 5.1]{GNT} that 
$$X \left( \mathbb F_{p^e} \right) \equiv W \left( \mathbb F_{p^e} \right) \equiv X_1 \left( \mathbb F_{p^e} \right) \mod p^e.$$ 
On the other hand, $X$ and $X_1$ are of Fano type over $Y$ and $Y_1$ respectively. 
Hence by \cite[Theorem 5.4]{GNT}, 
we obtain 
$$X \left( \mathbb F_{p^e} \right) \equiv Y \left(\mathbb F_{p^e} \right),\,\,\, X_1 \left( \mathbb F_{p^e} \right) \equiv Y_1 \left( \mathbb F_{p^e} \right) 
\mod p^e.$$
To summarise, we get  
$$Y \left( \mathbb F_{p^e} \right) \equiv Y' \left( \mathbb F_{p^e} \right) \mod p^e.$$
This completes the proof of Step~\ref{step-fp-bar}. 
\end{proof}

\begin{step}\label{step-alg-closed}
The assertion of Theorem \ref{t-conic-base} holds 
if $k$ is algebraically closed. 
\end{step}

\begin{proof}[Proof of Step~\ref{step-alg-closed}]
We fix a closed point $y \in Y$ and we may assume that $y$ is a unique singularity of $Y$. 
Let $g:Z \to Y$ be a log resolution such that $g \left( \Ex \left( g \right) \right)=y$. 
By Proposition \ref{p-rat-curve2}, 
$Y$ has $W\MO$-rational singularities if and only if $\Ex(g)$ is a tree 
of smooth rational curves. 
We take a model over 
some finitely generated $\overline{\mathbb{F}}_p$-algebra $R$ of a diagram $X \to Y \leftarrow Z$, 
i.e.\ an intermediate ring $\overline{\mathbb{F}}_p \subset R \subset k$ 
that is a finitely generated $\overline{\mathbb{F}}_p$-algebra, and 
$R$-morphisms of projective schemes over $R$ 
$$\mathfrak X \to \mathfrak Y \leftarrow \mathfrak Z$$
whose base changes by $(-)\times_R k$ are the same as $X \to Y \leftarrow Z$. 
Then, the base change $\mathfrak X_{\mu} \overset{f_{\mu}}\to \mathfrak Y_{\mu} \overset{g_{\mu}}\leftarrow \mathfrak Z_{\mu}$ 
by a general closed point $\mu \in \Spec\, R$ satisfies the same properties as $X \to Y \leftarrow Z$. 
Therefore, $\Ex \left( g_{\mu} \right)$ is a tree of smooth rational curves by Step~\ref{step-fp-bar}, 
hence so is $\Ex(g)$ by Proposition \ref{p-rat-curve2} and 
the upper-semi continuity of cohomologies \cite[Ch.\ III,\ Theorem 12.11]{Har77}. 
This completes the proof of Step~\ref{step-alg-closed}. 
\end{proof}

\begin{step}\label{step-general}
The assertion of Theorem \ref{t-conic-base} holds 
without any additional assumptions. 
\end{step}

\begin{proof}[Proof of Step~\ref{step-general}]
Thanks to Proposition~\ref{p-WO-bc}, 
we may assume that $k$ is algebraically closed. 
Then the assertion of Theorem \ref{t-conic-base} follows from Step~\ref{step-alg-closed}
\end{proof}

Step \ref{step-general} completes the proof of Theorem~\ref{t-conic-base}. 
\end{proof}

\begin{prop}\label{p-conic-van}
Let $k$ be a perfect field of characteristic $p>5$. 
Let $f:X \to Y$ be a projective $k$-morphism of normal $k$-varieties 
with $f_*\MO_X=\MO_Y$ 
which satisfies the following properties. 
\begin{enumerate}
\item{$\dim X=3$ and $\dim Y=2$.}
\item{There exists an effective $\R$-divisor $\Delta$ on $X$ such that 
$(X, \Delta)$ is klt and $- \left( K_X+\Delta \right)$ is $f$-nef and $f$-big.}
\end{enumerate}
Then $R^if_*(W\MO_{X, \Q})=0$ for all $i>0$. 
\end{prop}

\begin{proof}
By Lemma \ref{l-geom-normalise}, there exists a commutative diagram 
$$\begin{CD}
X' @>\alpha >> X\\
@VVf'V @VVf V\\
Y' @>\beta >> Y
\end{CD}$$
of projective $k$-morphisms of normal $k$-varieties 
which satisfies the properties listed in Lemma \ref{l-geom-normalise} 
(for an alternative argument, see Remark \ref{r-conic-van}). 
Since $(X, \Delta)$ is klt, $X$ has $W\MO$-rational singularities (Theorem \ref{t-gnt} (1)). 
By Theorem \ref{t-conic-base}, also $Y$ has $W\MO$-rational singularities. 
Since $\alpha$ and $\beta$ are finite universal homeomorphisms (Lemma \ref{l-geom-normalise} (1)), 
also $X'$ and $Y'$ have $W\MO$-rational singularities. 
Since the generic fibre of $f'$ is smooth and will be a rational curve after taking the base change to the algebraic closure, 
it follows from Theorem \ref{thm:CR481} that 
$Rf'_* \left( W\MO_{X', \Q} \right) \simeq W\MO_{Y', \Q}$. 
Therefore, we get 
$$Rf_* \left( W\MO_{X, \Q} \right) \simeq Rf_*R\alpha_* \left( W\MO_{X', \Q} \right)
\simeq R\beta_* \left( W\MO_{Y', \Q} \right) \simeq W\MO_{Y, \Q},$$
where the first and the last isomorphisms follow because 
$\alpha$ and $\beta$ are finite universal homeomorphisms (Lemma \ref{l-geom-normalise} (1)), 
and the second isomorphism follows from $Rf' \left( W\MO_{X', \Q} \right) \simeq W\MO_{Y', \Q}$. 
\end{proof} 

\begin{rem}\label{r-conic-van}
In the situation of Proposition \ref{p-conic-van}, 
the generic fibre is a conic curve in $\mathbb P^2_{K(Y)}$ (cf. \cite[Lemma 10.6(3)]{Kol13}). 
Hence, the assumption $p>2$ implies that 
the generic fibre of $f$ is generically smooth. 
Thus, $\alpha$ and $\beta$ in the proof can be assumed to be isomorphisms 
and we can avoid using Lemma \ref{l-geom-normalise}. 
We adopt the above argument, as it is less dependent on 
the assumption on the characteristic $p$. 
\end{rem}

\subsection{Proof of $W\MO$-vanishing for log Fano contractions} 

We now prove the main theorem of this section.

\begin{thm}\label{t-rel-van}
Let $k$ be a perfect field of characteristic $p > 5$. 
Let $f:X \to Y$ be a projective $k$-morphism of 
normal $k$-varieties. 
Assume that $\dim X \leq 3$ and 
there exists an effective $\mathbb{R}$-divisor $\Delta$ such that 
$(X, \Delta)$ is klt and $-(K_X+\Delta)$ is $f$-nef and $f$-big. 
Then $R^if_* \left( W\MO_{X, \Q} \right)=0$ for $i>0$. 
\end{thm}

\begin{proof}
Taking the Stein factorisation of $f$, we may assume that $f_* \mathcal{O}_X = \mathcal{O}_Y$. 
If $\dim Y=0$, then the assertion follows from Theorem \ref{t-gnt} (2). 
If $\dim Y=3$, then we have that also $(Y, \Delta_Y)$ is klt for some 
effective $\R$-divisor $\Delta_Y$, hence the assertion holds by Theorem \ref{t-gnt} (1). 
If $\dim Y=1$ (resp.\ $\dim Y=2$), then the assertion follows from  
Proposition~\ref{p-dP-fib} (resp.\ Proposition~\ref{p-conic-van}). 
\end{proof}

\section{A Nadel vanishing theorem for Witt multiplier ideal sheaves}\label{s-nadel}

In this section, we prove the main theorem of this paper (Theorem \ref{theorem:NWV}). 
Our strategy is to run a suitable minimal model program, 
which enables us to replace the given variety $X$ by the end result. 
In Subsection \ref{subsection:WVC_under_MMP}, we study the behaviour of Witt vector cohomologies 
under such minimal model programs. 
In Subsection \ref{subsection:MFS}, we prove Theorem \ref{theorem:NWV} 
for dlt Mori fibre spaces with an extra assumption (Lemma \ref{lem:MFS}). 
In Subsection \ref{subsection_reduction}, we give a proof of Theorem \ref{theorem:NWV}. 
Furthermore, we also give a generalisation of Theorem \ref{theorem:NWV} (Theorem \ref{theorem:NWV2}) and 
the Koll\'ar--Shokurov connectedness theorem (Theorem \ref{theorem:conn}). 

\subsection{Witt vector cohomologies under MMP}\label{subsection:WVC_under_MMP}

The purpose of this subsection is to prove the following.

\begin{prop}\label{prop:inv_mmp} 
Let $k$ be a perfect field of characteristic $p>5$. 
Let $(X, \Omega)$ be a three-dimensional $\Q$-factorial log pair over $k$ and 
let $h: X \to Z$ be a projective $k$-morphism to 
a quasi-projective $k$-scheme $Z$. 
Suppose that 
\begin{itemize}
\item $\left( X, \Omega ^{\wedge 1} \right)$ is dlt, and 
\item $K_X + \Omega \sim _{Z, \mathbb{R}} 0$. 
\end{itemize}
Let 
\[
X = :X_0 \dasharrow X_1 \dasharrow X_2 \dasharrow \cdots \dasharrow X_r \dasharrow \cdots
\]
be a $\left( K_X + \Omega^{\wedge 1} \right)$-MMP over $Z$ 
with the induced morphism $h_r:X_r \to Z$. 
Set $\Omega_r$ to be the push-forward of $\Omega$ on $X_r$. 
Then the isomorphism 
\[
R^i h_{*} WI_{\Omega^{\ge 1}, \mathbb{Q}} \simeq 
R^i (h_r)_* WI_{\Omega_r ^{\ge 1}, \mathbb{Q}}
\]
holds for any $i \ge 0$ and $r \geq 0$. 
\end{prop}

\begin{proof}
By induction, it is sufficient to prove the case when $r=1$. 
We have the following properties.  
\begin{enumerate}
\item $(X_1, \Omega_1)$ satisfies the same conditions as in the statement, 
i.e.\ $(X_1, \Omega_1)$ is a $\Q$-factorial log pair 
such that $\left( X_1, \Omega_1 ^{\wedge 1} \right)$ is dlt and 
$K_{X_1} + \Omega _1 \sim _{Z, \mathbb{R}} 0$. 
\item Given projective birational morphisms $\psi_0: Y \to X$ and $\psi_1:Y \to X_1$, 
it holds that $\psi_0^*(K_X + \Omega) \sim_{Z, \R} \psi_1^*(K_{X_1} + \Omega_1) \sim_{Z, \R} 0$. 
\item $\Nklt (X, \Omega) = \Supp \Omega^{\ge 1}$ and $\Nklt (X, \Omega_1) = \Supp \Omega _1 ^{\ge 1}$ (cf.\ Proposition \ref{p-nklt-non-negative}). 
\end{enumerate}

\vspace{2mm}

\noindent 
\underline{\textbf{Case 1.}}\ \ 
Suppose that $g: X \to X_1$ is a divisorial contraction.

We have the following spectral sequence 
\[
E^{i,j} _2 := R^i (h_1)_* R^j g_{*} \left( WI_{\Omega  ^{\ge 1}, \mathbb{Q}} \right) 
\Rightarrow  R^{i+j} \left( h_{1} \circ g \right)_* \left( WI_{\Omega^{\ge 1}, \mathbb{Q}} \right) =: E^{i+j}. 
\] 
Hence, it is sufficient to show the following two equations:   
\begin{equation}\label{e2-inv_mmp} 
g_{*} WI_{\Omega^{\ge 1}, \mathbb{Q}} = WI_{\Omega_{1} ^{\ge 1}, \mathbb{Q}}, 
\quad \text{and}
\end{equation}
\begin{equation}\label{e1-inv_mmp} 
R^i g_{*} WI_{\Omega  ^{\ge 1}, \mathbb{Q}} = 0\quad {\rm for} \quad i > 0. 
\end{equation}
Here, the equation (\ref{e2-inv_mmp}) is equivalent to the equation 
\begin{equation}\label{e3-inv_mmp} 
g \left( \Supp \left( \Omega ^{\ge 1} \right) \right) = \Supp \bigl( \Omega _1 ^{\ge 1} \bigr) 
\end{equation}
as sets. 
\vspace{2mm}

\noindent
\underline{\textbf{Case 1-1.}}\ 
Suppose that the contracted divisor $E$ is an irreducible component of $\Supp \left( \Omega ^{\ge 1} \right)$. 

In this case, the equation (\ref{e1-inv_mmp}) follows from \cite[Theorem 2.4]{BBE07}, since $g$ is an isomorphism outside $\Supp \left( \Omega  ^{\ge 1} \right)$. 
By (2), $g(E)$ is a non-klt centre of $(X_1, \Omega _1)$. 
Thanks to (3), it holds that $g (E) \subset \Supp \bigl( \Omega _1 ^{\ge 1} \bigr)$, 
which implies the required equation (\ref{e3-inv_mmp}).

\vspace{2mm}

\noindent
\underline{\textbf{Case 1-2.}}\ 
Suppose that the contracted divisor $E$ is not contained in $\Supp(\Omega ^{\ge 1})$. 

The equation (\ref{e3-inv_mmp}) is trivial in this case. 
We have the exact sequence: 
\[
0 \to WI_{\Omega ^{\ge 1}, \mathbb{Q}} \to W\mathcal{O}_{X, \mathbb{Q}} \to 
W\mathcal{O}_{\Supp \left( \Omega  ^{\ge 1} \right), \mathbb{Q}} \to 0. 
\]
In order to prove (\ref{e1-inv_mmp}), it is sufficient to show that 
\begin{itemize}
\item $0 \to g_* WI_{\Omega  ^{\ge 1}, \mathbb{Q}} \to g_{*} W\mathcal{O}_{X, \mathbb{Q}} \to 
g_{*} W\mathcal{O}_{\Supp \left( \Omega ^{\ge 1} \right), \mathbb{Q}} \to 0$ is exact, and
\item $R^i g_{*} W\mathcal{O}_{X, \mathbb{Q}} \simeq R^i g_{*} W\mathcal{O}_{\Supp \left( \Omega  ^{\ge 1} \right), \mathbb{Q}}$
holds for $i > 0$. 
\end{itemize}
Since 
\[
R^i g_{*} W\mathcal{O}_{X, \mathbb{Q}} \simeq 
\begin{cases}
W\mathcal{O}_{X_1, \mathbb{Q}} & (i=0) \\
0 & (i > 0)
\end{cases}
\]
holds by the $W\mathcal{O}$-rationality of the klt threefolds $X$ and $X_1$ (Theorem \ref{t-gnt} (1)), it is sufficient to show that 
\begin{align}
R^i g_{*} W\mathcal{O}_{\Supp \left( \Omega ^{\ge 1} \right), \mathbb{Q}} \simeq  
\begin{cases}
W\mathcal{O}_{\Supp \bigl( \Omega _1 ^{\ge 1} \bigr), \mathbb{Q}} & (i=0) \\
0 & (i > 0)
\end{cases}.
\end{align}
This follows from the 
Mayer--Vietoris exact sequence (Lemma \ref{l-MV})
\[
0 \to W \mathcal{O}_{S_1 \cup S_2, \mathbb{Q}} \to 
W \mathcal{O}_{S_1, \mathbb{Q}} \oplus W \mathcal{O}_{S_2, \mathbb{Q}}
\to W \mathcal{O}_{S_1 \cap S_2, \mathbb{Q}} \to 0
\]
for each union $S_i$ of strata of $\Supp \left( \Omega ^{\ge 1} \right)$ 
and Claim \ref{claim:divcontr} below. 
To summarise, in order to prove (\ref{e1-inv_mmp}), it suffices to show Claim \ref{claim:divcontr}. 

\begin{claim}\label{claim:divcontr}
If $S$ is a stratum of $\Supp \left( \Omega  ^{\ge 1} \right)$, then 
$g$ induces 
\[
R^i g_{*} W\mathcal{O} _{S, \mathbb{Q}} \simeq 
\begin{cases}
W\mathcal{O} _{g(S), \mathbb{Q}} & (i=0) \\
0 & (i > 0)
\end{cases}.
\]
\end{claim}

\begin{proof}[Proof of Claim \ref{claim:divcontr}]
Note that $S$ is normal since $\left( X, \Omega^{\wedge 1} \right)$ is dlt and $p > 5$ (cf.\ \cite[Proposition 4.1]{HX15}). 
We define an effective $\mathbb{R}$-divisor $\Lambda _S$ by adjunction $K_S + \Lambda_S = \left( K_X + \Omega ^{\wedge 1} \right)\big|_{S}$. 
Then $-(K_S + \Lambda _S)$ is $\left( g|_S \right)$-ample and $\left( S, \Lambda _S \right)$ is dlt. 
Hence $R^i g_{*} W\mathcal{O} _{S, \mathbb{Q}} = 0$ holds for $i>0$ by \cite[Proposition 3.3]{GNT}. 

In order to prove the required equation $g_{*} W\mathcal{O} _{S, \mathbb{Q}} = W\mathcal{O} _{g(S), \mathbb{Q}}$, 
it is sufficient to prove that $g: S \to g(S)$ has connected fibres (Lemma \ref{l-GNT2.22}). 
When $\dim S = 2$, then this follows because $g: S \to g(S)$ is 
a projective birational morphism of normal varieties. 
When $\dim S \le 1$, 
we can apply \cite[Lemma 3.10]{GNT} (cf.\ Remark \ref{r-GNT3.10}). 
\end{proof}

\vspace{2mm}

\noindent 
\underline{\textbf{Case 2.}}\ \ 
Suppose that $g: X \to Z'$ is a $\left( K_X + \Omega ^{\wedge 1} \right)$-flipping contraction over $Z$, and let $g_1:X_1 \to Z'$ be its flip.

\begin{claim}\label{claim:flip}
\begin{itemize}
\item[(\ref{claim:flip}.1)] $g_{*} WI_{\Omega  ^{\ge 1}, \mathbb{Q}} \simeq  
(g_1)_{*} WI_{\Omega _1 ^{\ge 1}, \mathbb{Q}}$ holds, and
\item[(\ref{claim:flip}.2)] $R^i g_{*} WI_{\Supp \left( \Omega  ^{\ge 1} \right), \mathbb{Q}} = 0$ and 
$R^i (g_1)_{*}  WI_{\Supp \bigl( \Omega_1  ^{\ge 1} \bigr), \mathbb{Q}} = 0$ hold for any $i > 0$. 
\end{itemize}
\end{claim}

\begin{proof}[Proof of Claim \ref{claim:flip}]
First, (\ref{claim:flip}.1) follows from the set-theoretical 
equation $g \left( \Supp \left( \Omega  ^{\ge 1} \right) \right) = g_1 \left( \Supp \bigl( \Omega _1 ^{\ge 1} \bigr) \right)$, which is trivial. 

Let us prove (\ref{claim:flip}.2). 
We may assume that $i=1$, 
since both $g$ and $g_1$ have at most one-dimensional fibres (\cite[Lemma 2.20]{GNT}). 
Consider the exact sequences
\begin{align*}
0 \to WI_{\Omega  ^{\ge 1}, \mathbb{Q}} &\to W\mathcal{O}_{X, \mathbb{Q}} \to W\mathcal{O}_{\Supp \left( \Omega  ^{\ge 1} \right), \mathbb{Q}} \to 0, \quad\text{and}\\
0 \to WI_{\Omega _1  ^{\ge 1}, \mathbb{Q}} &\to W\mathcal{O}_{X_1, \mathbb{Q}} \to W\mathcal{O}_{\Supp \bigl( \Omega _1  ^{\ge 1} \bigr), \mathbb{Q}} \to 0. 
\end{align*}
We have that $R^1 g_* W\mathcal{O}_{X, \mathbb{Q}} = 0$ and $R^1 (g_1)_*  W\mathcal{O}_{X_1, \mathbb{Q}} = 0$, 
since all of $X$, $X_1$ and $Z'$ have $W \mathcal{O}$-rational singularities (Theorem \ref{t-gnt} (1)). 
Hence, it is sufficient to show the surjectivity of 
\begin{align*}
g_* W\mathcal{O}_{X, \mathbb{Q}} &\to g_* W\mathcal{O}_{\Supp \left( \Omega  ^{\ge 1} \right), \mathbb{Q}}, \quad\text{and} \\
(g_1)_* W\mathcal{O}_{X_1, \mathbb{Q}} &\to (g_1)_* W\mathcal{O}_{\Supp \bigl( \Omega _1  ^{\ge 1} \bigr), \mathbb{Q}}, 
\end{align*}
which are equivalent to 
\begin{align*}
g_* W\mathcal{O}_{\Supp \left( \Omega  ^{\ge 1} \right), \mathbb{Q}} &= W\mathcal{O}_{g \left( \Supp \left( \Omega  ^{\ge 1} \right) \right), \mathbb{Q}}, \quad\text{and}\\
(g_1)_* W\mathcal{O}_{\Supp \bigl( \Omega_1  ^{\ge 1} \bigr), \mathbb{Q}} &= W\mathcal{O}_{g_1 \bigl( \Supp \bigl( \Omega _1  ^{\ge 1} \bigr) \bigr), \mathbb{Q}} 
\end{align*}
respectively. 
Therefore, by Lemma \ref{l-GNT2.22}, it is enough to prove the following: 
\begin{itemize}
\item[(i)] $g': \Supp \left( \Omega ^{\ge 1} \right) \to Z'$ has connected fibres. 
\item[(ii)] $g'_1: \Supp \bigl( \Omega _1  ^{\ge 1} \bigr) \to Z'$ has connected fibres. 
\end{itemize}
Both (i) and (ii) follow from (1) and Theorem \ref{t-birat-connected2}. 
This completes the proof of Claim \ref{claim:flip}. 
\end{proof}

For the induced morphism $\theta:Z' \to Z$, 
we obtain isomorphisms 
\begin{align*}
Rh_*WI_{\Omega ^{\ge 1}, \mathbb{Q}} 
&\simeq R\theta_*Rg_* \left( WI_{\Omega  ^{\ge 1}, \Q} \right) \\ 
&\simeq
R\theta_*R(g_1)_* \left( WI_{\Omega_1 ^{\ge 1}, \Q} \right) \simeq 
R(h_1)_* \left( WI_{\Omega_1 ^{\ge 1}, \Q} \right),
\end{align*}
where the second isomorphism follows from Claim \ref{claim:flip}. 
This completes the proof of Proposition \ref{prop:inv_mmp}. 
\end{proof}

\begin{rem}\label{r-GNT3.10}
In the proof above, we use \cite[Lemma 3.10]{GNT}, 
which is a special case of the two-dimensional version of Theorem \ref{intro-NWV}. 
In the proof of \cite[Lemma 3.10]{GNT}, they use \cite[Proposition 2.2]{Tan16}, 
whose proof depends on a classification result on surfaces. 
Here, for the reader's convenience, we give a sketch of an alternative proof. 
When $U$ in \cite[Lemma 3.10]{GNT} has positive dimension, the assertion follows from
the Nadel vanishing theorem for two-dimensional relative cases (\cite[Theorem 2.10]{Tan15}). 
Hence, the remaining case is when $(S, \Delta_S)$ is a two-dimensional dlt pair 
over an algebraically closed field such that $-(K_S+\Delta_S)$ is ample, 
and it is sufficient to show that $\llcorner \Delta_S \lrcorner$ is connected. 
In this case, we may apply the idea of this paper (cf.\ (B) of Subsection \ref{ss1-intro}), 
and the problem can be reduced to the study on Mori fibre spaces (cf.\ Subsection \ref{subsection:MFS}). 
\end{rem}

\subsection{Vanishing for Mori fibre spaces}\label{subsection:MFS}

In this subsection, we prove Lemma \ref{lem:MFS}, which is a special case of Theorem \ref{theorem:NWV}. 
We start with the following auxiliary result. 

\begin{lem}\label{l-slc-curve}
Let $k$ be a perfect field of characteristic $p>0$. 
Let $(S, \Delta_S)$ be a two-dimensional dlt pair over $k$ and 
let $f:S \to Z$ be a projective $k$-morphism to a quasi-projective $k$-scheme $Z$. 
Assume that $- \left( K_S+\Delta_S \right)$ is $f$-ample. 
Let $h:\llcorner \Delta_S\lrcorner \to Z$ be the induced morphism. 
Then the following holds. 
\begin{enumerate}
\item $R^ih_*\MO_{\llcorner \Delta_S\lrcorner}=0$ for $i>0$. 
\item $R^ih_* \left( W\MO_{\llcorner \Delta_S\lrcorner, \Q} \right)=0$ for $i>0$. 
\end{enumerate}
\end{lem}

\begin{proof}
Let us prove (1). 
Taking the Stein factorisation of $f$, we may assume that $f_*\MO_S=\MO_Z$. 
Furthermore, the problem is reduced to the case when $k$ is algebraically closed. 
Thanks to \cite[Theorem 2.12 and Theorem 3.5]{Tan15}, it holds that 
$R^if_*\MO_S=0$ for $i>0$. 
If $\dim Z \geq 1$, then we have a surjection: 
\[
0=R^1f_*\MO_S \to R^1h_*\MO_{\llcorner \Delta_S\lrcorner}, 
\]
which implies $R^ih_*\MO_{\llcorner \Delta_S\lrcorner}=0$ for $i>0$. 
Thus, we may assume that $\dim Z=0$, i.e.\ $Z=\Spec\,k$. 
Then we have the exact sequence: 
\[
0=H^1(S, \MO_S) \to H^1 \left( \llcorner \Delta_S\lrcorner, \MO_{\llcorner \Delta_S\lrcorner} \right) 
\to H^2 \left( S, \MO_S \left( -\llcorner \Delta_S\lrcorner \right) \right).
\]
Hence, it suffices to prove $H^2 \left( S, \MO_S \left( -\llcorner \Delta_S\lrcorner \right)\right)=0$. 

It follows from Serre duality that 
\[
h^2 \left( S, \MO_S \left( -\llcorner \Delta_S\lrcorner \right) \right) =h^0 \left( S, \MO_S \left (K_S+\llcorner \Delta_S\lrcorner \right) \right).
\]
For an ample Cartier divisor $A$, we have that 
\[
\left( K_S+\llcorner \Delta_S\lrcorner \right) \cdot A= \left( \left( K_S+\Delta_S \right) -\{\Delta_S\} \right) \cdot A<0, 
\]
which in turn implies that $H^0 \left( S, \MO_S \left( K_S+\llcorner \Delta_S\lrcorner \right) \right)=0.$ 
Hence, it holds that $H^2 \left( S, \MO_S \left( -\llcorner \Delta_S\lrcorner \right) \right)=0$, as desired. 
This completes the proof of (1).

The assertion (2) follows from (1) and \cite[Lemma 2.19]{GNT}. 
\end{proof}

\begin{lem}\label{lem:MFS2}
Let $k$ be a perfect field of characteristic $p>5$. 
Let $(X, \Xi)$ be a three-dimensional $\Q$-factorial dlt pair over $k$ and 
let $f: X \to Z$ be a projective surjective $k$-morphism to a  quasi-projective $k$-scheme $Z$ such that 
\begin{enumerate}
\item 
$\dim X >\dim Z$, 
\item 
$f$ has connected fibres, 
\item 
$-(K_X+\Xi)$ is $f$-ample, and 
\item 
there exists an irreducible component $D_0$ of $\llcorner \Xi\lrcorner$ 
such that $D_0$ is $f$-ample. 
\end{enumerate}
Then 
\[
R^i f_* W\mathcal{O}_{\Supp \lfloor \Xi \rfloor, \mathbb{Q}} \simeq  
\begin{cases}
W\mathcal{O}_{Z, \mathbb{Q}} & (i=0) \\
0 & (i > 0)
\end{cases}
\]
holds.
\end{lem}

\begin{proof}
Replacing $Z$ by $Z'$ for the Stein factorisation $X \to Z' \to Z$ of $f$, 
we may assume that $f_*\MO_X=\MO_Z$. 
Let $g:D_0 \to Z$ be the induced morphism. 
Let $\llcorner \Xi \lrcorner=\sum_{i=0}^m D_i$ be the irreducible decomposition.

\setcounter{step}{0}

\begin{step}\label{s1-MFS2}
The isomorphism 
$R g_* W\mathcal{O}_{D_0, \mathbb{Q}} \simeq
W\mathcal{O}_{Z, \mathbb{Q}}$ holds. 
In particular, Lemma \ref{lem:MFS2} holds if $m=0$.  
\end{step}

\begin{proof}[Proof of Step \ref{s1-MFS2}] 
We define an effective $\mathbb{R}$-divisor 
$\Xi_{D_0}$ by adjunction $(K_X+\Xi)|_{D_0}=K_{D_0}+\Xi_{D_0}$. 
It holds that $\left( D_0, \Xi_{D_0} \right)$ is dlt and $-\left( K_{D_0}+\Xi_{D_0} \right)$ is $g$-ample. 
Hence it follows from \cite[Proposition 3.3]{GNT} that 
$R^i g_* W\mathcal{O}_{D_0, \mathbb{Q}} = 0$ for $i > 0$. 

In order to prove $g _* W\mathcal{O}_{D_0, \mathbb{Q}}\simeq W\mathcal{O}_{Z, \mathbb{Q}}$, 
it suffices to show that $g:D_0 \to Z$ has connected fibres (Lemma \ref{l-GNT2.22}). 
Since $Z$ is normal, it is enough to prove that $D_0|_F$ is geometrically connected 
for the generic fibre $F:=X \times_Z \Spec\,K(Z)$ of $f$. 
If $\dim F \geq 2$, then the restriction $D_0|_F$ is geometrically connected, 
because $D_0|_F$ is an ample $\mathbb{Q}$-Cartier $\Z$-divisor on $F$. 
Thus, we may assume that $\dim F=1$. 
Since $-(K_X+\Xi)$ is $f$-ample, the geometric generic fibre 
$\overline F:=F \times_{K(Z)} \overline{K(Z)}$ is $\mathbb P^1_{\overline{K(Z)}}$. 
Moreover, it holds that 
\[
0>\deg (K_X+\Xi)|_{\overline F} \geq \deg \left( K_{\overline F}+D_0|_{\overline F} \right)=-2+\deg \left( D_0|_{\overline F} \right).
\]
This implies that $g$ has connected fibres. 
This completes the proof of Step \ref{s1-MFS2}. 
\end{proof}

\begin{step}\label{s2-MFS2} 
For any $i \in \{1, \ldots, m\}$, 
it holds that 
\begin{enumerate}
\item[(a)] $\dim f(D_i) < \dim D_i$, and
\item[(b)] $f(D_i) =f(D_i \cap D_0).$ 
\end{enumerate}
\end{step}

\begin{proof}[Proof of Step \ref{s2-MFS2}] 
Let us prove (a). 
If $\dim Z \leq 1$, then there is nothing to show. 
Thus we may assume that $\dim Z = 2$. 
Assuming that there is $i \in \{1, \dots, m\}$ such that 
such that $f(D_i)=Z$, let us derive a contradiction. 
Since a general fibre $F$ of $f$ is $\mathbb P^1$, it holds that 
\[
(K_X+\Xi) \cdot F \geq (K_X+D_0+D_i) \cdot F \geq 0.
\]
This contradicts the fact that $-(K_X+\Xi)$ is $f$-ample. 
Therefore, $D_i$ does not dominate $Z$ for any $i \in \{1, \dots, m\}$, which implies (a).

Let us prove (b). 
Fix an arbitrary closed point $x \in f(D_i)$. 
Since $\dim f(D_i) < \dim D_i$, 
there exists a curve $C$ on $X$ contained in $D_i \cap f^{-1}(x)$. 
Since $D_0$ is ample over $Z$, 
the contracted curve $C$ intersects $D_0$. 
This implies $x \in f(D_i \cap D_0)$. 
Thus we get $f(D_i) = f(D_i \cap D_0)$. 
Hence, (b) holds. 
This completes the proof of Step \ref{s2-MFS2}. 
\end{proof}

\begin{step}\label{s3-MFS2} 
For any $i \in \{1, \ldots, m\}$, 
it holds that the induced homomorphism
\[
R^q f _* W\mathcal{O}_{D_i, \mathbb{Q}} \to R^q f _* W\mathcal{O}_{D_i \cap E, \mathbb{Q}},
\]
is an isomorphism for any $q \geq 0$, where $E:=\bigcup_{j \in \{0, \ldots, m\} \setminus \{i\}} D_j$. 
\end{step}

\begin{proof}[Proof of Step \ref{s3-MFS2}] 
Set $C:= f(D_i)$, and 
let $D_i \xrightarrow{f'} C' \xrightarrow{s} C$ be the 
Stein factorisation of $D_i \to C$. 
Let $\Xi_{D_i}$ be the effective $\R$-divisor on $D_i$ defined by adjunction 
$(K_X+\Xi)|_{D_i}=K_{D_i}+\Xi_{D_i}$. 
Then the following properties hold. 
\begin{enumerate}
\item[(c)] $(D_i, \Xi_{D_i})$ is dlt and $\Supp \left( \lfloor \Xi _{D_i} \rfloor \right) = D_i \cap E$.
\item[(d)] $- \left( K_{D_i} + \Xi_{D_i} \right)$ is $f'$-ample. 
\end{enumerate}
We have the exact sequence: 
\[
0 \to WI_{\llcorner \Xi_{D_i} \lrcorner, \Q} \to W\MO_{D_i, \Q} \to W\MO_{D_i \cap E, \Q} \to 0.
\]
By \cite[Proposition 3.3]{GNT} and Lemma \ref{l-slc-curve}, it holds that 
\[
R^if'_* W\mathcal{O}_{D_i, \mathbb{Q}} =0 \quad {\rm and}\quad 
R^if'_* W\mathcal{O}_{D_i \cap E, \mathbb{Q}} =0
\]
for $i>0$, respectively. 
Hence, it suffices to prove that $f'_*W\MO_{D_i, \Q} \to f'_*W\MO_{D_i \cap E, \Q}$ 
is an isomorphism. 
By Step \ref{s2-MFS2} (b), it is enough to prove that 
the induced morphism $D_i \cap E \to C'$ has connected fibres (Lemma \ref{l-GNT2.22}), 
which follows from \cite[Lemma 3.10]{GNT}. 
This completes the proof of Step \ref{s3-MFS2}. 
\end{proof}

\begin{step}\label{s4-MFS2}
The assertion of Lemma \ref{lem:MFS2} holds.   
\end{step}

\begin{proof}[Proof of Step \ref{s4-MFS2}] 
We prove the assertion by induction on $m$. 
By Step \ref{s1-MFS2}, there is nothing to show if $m=0$. 
Thus, assume that $m>0$ and 
that the assertion of Lemma \ref{lem:MFS2} holds 
if the number of the irreducible components of $\llcorner \Xi\lrcorner$ is less than $m$. 
Fix $i \in \{1, \ldots, m\}$.  
Since $(X, \Xi':=\Xi - \epsilon D_i)$ satisfies the same assumption as in Lemma \ref{lem:MFS2} for sufficiently small $\epsilon >0$, 
it follows from the induction hypothesis that
\[
R f_* W\mathcal{O}_{E, \mathbb{Q}} \simeq 
W\mathcal{O}_{Z, \mathbb{Q}},
\] 
where $E:=\bigcup_{j \in \{0, \ldots, m\} \setminus \{i\}} D_j$. 
By the Mayer--Vietoris exact sequence (Lemma \ref{l-MV})
\[
0 \to W\mathcal{O}_{D_i \cup E, \mathbb{Q}} \to 
W\mathcal{O}_{D_i, \mathbb{Q}} \oplus W\mathcal{O}_{E, \mathbb{Q}} \to 
W\mathcal{O}_{D_i \cap E, \mathbb{Q}} \to 0, 
\]
it is sufficient to show that the induced homomorphism
\[
R^q f _* W\mathcal{O}_{D_i, \mathbb{Q}} \to R^q f _* W\mathcal{O}_{D_i \cap E, \mathbb{Q}}. 
\]
is an isomorphism for any $q \geq 0$. 
This is nothing but the assertion of Step \ref{s3-MFS2}. 
This completes the proof of Step \ref{s4-MFS2}. 
\end{proof}
Step \ref{s4-MFS2} completes the proof of Lemma \ref{lem:MFS2}. 
\end{proof}

\begin{lem}\label{lem:MFS} 
Let $k$ be a perfect field of characteristic $p>5$. 
Let $(X, \Xi)$ be a three-dimensional $\Q$-factorial dlt pair over $k$ and 
let $f: X \to Z$ be a $(K_X+\Xi)$-Mori fibre space to a quasi-projective $k$-variety $Z$. 
Suppose that $f \left( \Supp \lfloor \Xi \rfloor \right)=Z$. 
Then 
\[
R^i f _* WI_{\llcorner \Xi \lrcorner, \mathbb{Q}} = 0
\]
for any $i \geq  0$. 
\end{lem}

\begin{proof}
It follows from Theorem \ref{t-rel-van} that 
\[
R f_* W\mathcal{O}_{X, \mathbb{Q}} \simeq W\mathcal{O}_{Z, \mathbb{Q}}. 
\]
By $f \left( \Supp \lfloor \Xi \rfloor \right)=Z$, there exists an irreducible component $D_0$ of $\llcorner \Xi\lrcorner$ such that $f(D_0)=Z$. 
Since $\rho(X/Z)=1$, we have that $D_0$ is $f$-ample. 
In particular, we may apply Lemma \ref{lem:MFS2} and obtain the isomorphism 
\[
R f_* W\mathcal{O}_{\Supp \llcorner \Xi \lrcorner, \mathbb{Q}} \simeq 
W\mathcal{O}_{Z, \mathbb{Q}}. 
\]
Therefore, by the exact sequence
\[
0 \to WI_{\llcorner \Xi \lrcorner, \mathbb{Q}} \to 
W\mathcal{O}_{X, \mathbb{Q}} \to W\mathcal{O}_{\Supp \llcorner \Xi \lrcorner, \mathbb{Q}}
\to 0, 
\]
the induced homomorphism 
\[
R^if_* W\mathcal{O}_{X, \mathbb{Q}} \to 
R^if_* W\mathcal{O}_{\Supp \llcorner \Xi \lrcorner, \mathbb{Q}}
\]
is an isomorphism for any $i \geq 0$. 
Therefore, $R^i f _* WI_{\llcorner \Xi \lrcorner, \mathbb{Q}} = 0$ for any $i \geq  0$. 
\end{proof}

\subsection{Proof of the main theorem and related results}\label{subsection_reduction}

In this subsection, we prove the main theorem of this paper (Theorem \ref{theorem:NWV}) and a generalisation of it (Theorem \ref{theorem:NWV2}). 
As a consequence, we obtain the Koll\'{a}r--Shokurov connectedness theorem 
(Theorem \ref{theorem:conn}).

\begin{lem}\label{l-dlt-modif2}
Let $k$ be a perfect field of characteristic $p > 5$. 
Let $(X, \Delta)$ be a three-dimensional quasi-projective log pair over $k$. 
Let $f:Y \to X$ be a projective birational morphism 
that satisfies the properties (1)-(3) of Proposition \ref{p-dlt-modif}. 
Set $\Delta_Y$ to be the effective $\mathbb{R}$-divisor defined by $K_Y+\Delta_Y=f^*(K_X+\Delta)$. 
Then the following hold. 
\begin{enumerate}
\item[(a)] 
$f_* WI_{\mathrm{Nklt}(Y, \Delta _Y), \mathbb{Q}} = WI_{\mathrm{Nklt}(X, \Delta _X), \mathbb{Q}}$ holds. 
\item[(b)] 
$R^i f_* WI_{\mathrm{Nklt}(Y, \Delta _Y), \mathbb{Q}} = 0$ holds for $i > 0$.
\end{enumerate} 
\end{lem}

\begin{proof}
The condition (a) follows from the equation 
$f(\mathrm{Nklt}(Y, \Delta _Y)) = \mathrm{Nklt}(X, \Delta)$ (Lemma \ref{l-nklt-crep}). 

Let us prove (b). 
By the $W\mathcal{O}$-rationality of klt threefolds (Theorem \ref{t-gnt} (1)), 
the vanishing in (b) holds outside $\mathrm{Nklt}(X, \Delta)$. 
The set of the non-$\mathbb{Q}$-factorial points 
on a three-dimensional klt pair is a zero-dimensional closed subset (\cite[Proposition 2.15 (4)]{GNT}). 
Thus, after removing finitely many closed points of $X \setminus \Nklt(X, \Delta)$, 
we may assume that all the non-$\mathbb{Q}$-factorial points of $X$ are contained in $\Nklt (X, \Delta)$. 
Hence it follows that $\Ex(f) \subset f^{-1} \left( \Nklt (X, \Delta) \right)$. 
Then we get 
\[
\Ex(f) \subset f^{-1} \left( \Nklt \left( X, \Delta \right) \right) = \Nklt (Y, \Delta_Y),
\]
where the equality holds by Proposition \ref{p-dlt-modif} (3). 
Since $f$ is an isomorphism outside $\Nklt (Y, \Delta_Y)$, 
it follows from \cite[Proposition 2.23]{GNT}  
that $R^i f_* WI_{\mathrm{Nklt}(Y, \Delta _Y), \mathbb{Q}} = 0$ holds for $i > 0$. 
This completes the proof of (b). 
\end{proof}

\begin{prop}\label{p-NWV-Omega}
Let $k$ be a perfect field of characteristic $p > 5$. 
Let $(X, \Omega)$ be a three-dimensional $\Q$-factorial log pair over $k$ and 
let $f:X \to Z$ be a projective $k$-morphism to a quasi-projective $k$-scheme $Z$. 
Assume that 
\begin{enumerate}
\item $\left( X, \Omega^{\wedge 1} \right)$ is dlt, 
\item $K_X+\Omega \sim_{f, \mathbb{R}} 0$, 
\item $\Omega$ is $f$-big, and 
\item $\Supp \Omega^{>1}=\Supp \Omega^{\geq 1}$. 
\end{enumerate}
Then $R^if_* \left( WI_{\Nklt \left(X, \Omega \right), \Q} \right) = 0$ for $i > 0$. 
\end{prop}

\begin{proof}
Taking the Stein factorisation of $f$, we may assume that $f_*\MO_X=\MO_Z$.
 We have that 
\begin{equation}\label{e1-NWV-Omega0}
\Supp \left( \Omega-\Omega^{\wedge 1} \right)= \Supp \Omega^{>1}= \Supp \Omega^{\geq 1}
=\Nklt(X, \Omega),
\end{equation}
where the second equality holds by (4) and 
the third one follows from Proposition \ref{p-nklt-non-negative}. 

\setcounter{step}{0}
\begin{step}\label{s1-NWV-Omega}
The assertion of Proposition \ref{p-NWV-Omega} holds, 
if there is a $\left( K_X+\Omega^{\wedge 1} \right)$-Mori fibre space $g:X \to Z'$ over $Z$. 
\end{step}

\begin{proof}[Proof of Step \ref{s1-NWV-Omega}] 
We have the induced morphisms: 
\[
f:X \xrightarrow{g} Z' \xrightarrow{h} Z.
\]
Since $\Omega - \Omega ^{\wedge 1}$ is $g$-ample, 
it follows from (\ref{e1-NWV-Omega0}) that $g \left( \Supp \Omega^{\geq 1} \right) = Z'$. 
Therefore, Lemma \ref{lem:MFS} implies $Rg_* \left( WI_{\Nklt \left( X, \Omega \right), \Q} \right)=0$. 
Hence, we have that 
$$Rf_* \left( WI_{\Nklt \left( X, \Omega \right), \Q} \right) \simeq Rh_*Rg_* \left( WI_{\Nklt \left(X, \Omega \right), \Q} \right)=0.$$ 
This completes the proof of Step \ref{s1-NWV-Omega}. 
\end{proof}

\begin{step}\label{s2-NWV-Omega}
In order to prove the assertion of Proposition \ref{p-NWV-Omega}, 
it is sufficient to prove the assertion under the following additional assumptions: 
\begin{enumerate}
\item[(i)] $\left( K_X+\Omega^{\wedge 1} \right)$ is $f$-nef, 
\item[(ii)] $(X, \Omega)$ is not klt, i.e. $\Supp \left( \Omega ^{\ge 1} \right) \neq \emptyset$, 
\item[(iii)] the set-theoretic equation $\Nklt(X, \Omega) = f^{-1}(Z_1)$ holds for some closed subset $Z_1$ of $Z$,
\item[(iv)] if we set $Z_0:=Z \setminus Z_1$ and $X_0:=f^{-1}(Z_0)$, 
then $X_0$ is of Fano type over $Z_0$, and 
\item[(v)] $\dim Z=1$. 
\end{enumerate}
\end{step}

\begin{proof}[Proof of Step \ref{s2-NWV-Omega}] 
By Theorem \ref{t-lc-MMP}, 
there exists a $\left( K_X+\Omega^{\wedge 1} \right)$-MMP over $Z$ that terminates:
\[
X=:X_0 \dashrightarrow X_1  \dashrightarrow \cdots \dashrightarrow X_{\ell}.
\]
Let $f_i:X_i \to Z$ be the induced morphism and 
let $\Omega_i$ be the push-forward of $\Omega$ on $X_i$. 
Then $(X_{\ell}, \Omega_{\ell})$ still satisfies the conditions (1)-(4) in Proposition \ref{p-NWV-Omega}. 
By Proposition \ref{prop:inv_mmp}, replacing $(X, \Omega)$ by $(X_{\ell}, \Omega _{\ell})$,  
we may assume that $\left( K_X+\Omega^{\wedge 1} \right)$ is $f$-nef or 
that there is a $\left( K_X+\Omega^{\wedge 1} \right)$-Mori fibre space $g:X \to Z'$ over $Z$. 
In the latter case, the assertion of Proposition \ref{p-NWV-Omega} follows from  Step \ref{s1-NWV-Omega}. 
Thus, we may assume that (i) holds. 
If $(X, \Omega)$ is klt, then (2) and (3) imply that $X$ is of Fano type over $Z$. 
In this case, the assertion of Proposition \ref{p-NWV-Omega} follows from Theorem \ref{t-rel-van}. 
Thus, we may assume (ii).

Since $K_X+\Omega^{\wedge 1} \sim_{f, \mathbb{R}} - \left( \Omega-\Omega^{\wedge 1} \right)$, 
it holds that $- \left( \Omega-\Omega^{\wedge 1} \right)$ is $f$-nef. 
Since $\Omega-\Omega^{\wedge 1}$ is a nonzero effective divisor 
by (4) and (ii), we have that $\dim Z \geq 1$. 
If $z \in Z$ is a closed point such that $f^{-1}(z) \cap \Supp \left( \Omega-\Omega^{\wedge 1} \right) \neq\emptyset$, 
then it follows that $f^{-1}(z) \subset \Supp \left( \Omega-\Omega^{\wedge 1} \right)$ 
because $- \left( \Omega-\Omega^{\wedge 1} \right)$ is $f$-nef. 
Thus, there exists a closed subset $Z_1$ of $Z$ such that $Z_1 \subsetneq Z$ and 
\[
\Nklt(X, \Omega)=\Supp \left( \Omega-\Omega^{\wedge 1} \right) = f^{-1}(Z_1),
\]
where the first equality follows from (\ref{e1-NWV-Omega0}). 
Thus (iii) holds. 
It is clear that (iii) implies (iv). 

Therefore, assuming $\dim Z \neq 1$, it is enough to show the assertion of Proposition \ref{p-NWV-Omega}. 
Since $\dim Z \geq 1$, we have $\dim Z \geq 2$. 
Then there are finitely many closed points $z_1, \ldots, z_n$ of $Z_0$ such that 
$R^i f_*\MO_{X}|_{Z_0 \setminus \{z_1, \ldots, z_n\}}=0$  
for $i>0$. 
Indeed, if $\Supp \left( R^i f_* \mathcal{O}_X \right) \cap Z_0$ contains a curve $C$, then 
it contradicts the vanishing obtained by \cite[Theorem 3.3]{Tanc} for  
the morphism $X \times _{Z} \Spec \mathcal{O}_{Z, \xi _C} \to \Spec \mathcal{O}_{Z, \xi _C}$, 
where $\xi_C$ denotes the generic point of $C$.  
By (iii), \cite[Proposition 4.6.1]{CR12}, and 
$R^i f_*\MO_{X}|_{Z_0 \setminus \{z_1, \ldots, z_n\}}=0$, 
it holds that 
$$R^if_* \left( WI_{\Nklt \left( X, \Omega \right), \Q} \right) \big| _{Z \setminus \{z_1, \ldots, z_n\}} = 0$$ 
for any $i > 0$. 
Since $X_0$ is of Fano type over $Z_0$, it follows from Theorem \ref{t-rel-van} that 
$$R^if_* \left( WI_{\Nklt \left( X, \Omega \right), \Q} \right) \big|_{Z_0} = R^i (f|_{X_0})_* \left( W \mathcal{O}_{X_0, \Q} \right) = 0$$ 
holds for any $i > 0$. 
Since $Z = Z_0 \cup \left( Z \setminus \{ z_1, \ldots, z_n \} \right)$, we obtain 
\[
R^if_* \left( WI_{\Nklt \left( X, \Omega \right), \Q} \right)=0
\]
for any $i>0$. 
Thus, the assertion of Proposition \ref{p-NWV-Omega} holds. 
This completes the proof of Step \ref{s2-NWV-Omega}. 
\end{proof}

\begin{step}\label{s3-NWV-Omega} 
Assume (i)--(v) in Step \ref{s2-NWV-Omega}. 
There exists a commutative diagram of projective morphisms of normal varieties: 
\begin{equation}\label{e2-NWV-Omega}
\begin{CD}
Y' @>\beta >> Y\\
@VVg'V @VVgV\\
X' @>\alpha >> X\\
@VVf'V @VVfV\\
Z' @>\gamma >> Z
\end{CD}
\end{equation}
such that 
\begin{enumerate}
\item[(vi)] $f'_*\MO_{X'}=\MO_{Z'}$,  
\item[(vii)] 
$\alpha, \beta$, and $\gamma$ are finite universal homeomorphisms, 
\item[(viii)] 
$g$ is a log resolution of $(X, \Omega)$, $g'$ is birational, and 
\item[(ix)] 
there are finitely many closed points $z_1, \ldots, z_n$ of $Z_0$ such that 
$$R^i h'_*\MO_{Y'}|_{\gamma^{-1} \left( Z_0 \setminus \left \{z_1, \ldots, z_n \right \} \right)}=0$$ 
for $i>0$, where $h':=f' \circ g'$. 
\end{enumerate} 
We set $h := f \circ g$ for later use. 
\end{step}

\begin{proof}[Proof of Step \ref{s3-NWV-Omega}] 
Set $\mathcal X:=X_{K(Z)}$. 
There exist a finite purely inseparable extension $K(Z) \subset L$ 
and a projective birational $L$-morphism $\mathfrak g'_1:\mathcal Y'_1 \to \mathcal X'$ 
of projective normal surfaces over $L$ 
such that 
there is a finite universal homeomorphism $\mathfrak a:\mathcal X' \to \mathcal X$, 
$\mathcal X' \times_L \overline L$ is isomorphic to the normalisation of 
$\bigl( \mathcal X \times_{K(Z)} \overline{K(Z)} \bigr)_{\red}$, 
and $\mathfrak g'_1 \times_L \overline L:\mathcal Y'_1\times_L \overline L \to \mathcal X' \times_L \overline L$ 
is a resolution of singularities of $\mathcal X' \times_L \overline L$
 (cf.\ Lemma \ref{l-geom-normalise}).  
In particular, $\mathcal Y'_1$ is smooth over $L$. 
There exists a normal projective surface $\mathcal Y_1$ over $K(Z)$ 
which completes the following commutative diagram: 
\begin{equation}\label{e3-NWV-Omega}
\begin{CD}
\mathcal Y'_1 @>\mathfrak b_1 >> \mathcal Y_1\\
@VV\mathfrak g_1'V @VV\mathfrak g_1V\\
\mathcal X' @>\mathfrak a>> \mathcal X\\
@VVV @VVV\\
\Spec\,L @>>> \Spec\,K(Z),
\end{CD}
\end{equation}
where $\mathfrak g_1$ is a projective birational morphism and $\mathfrak b_1$ is a 
finite universal homeomorphism. 
Indeed, we have $K\left( \mathcal Y'_1 \right) ^{p^e}=K \left( \mathcal X' \right) ^{p^e} \subset K(\mathcal X)$ for sufficiently large $e \in \Z_{>0}$, 
hence we can find such $\mathcal Y_1$ by taking the normalisation 
of $\mathcal Y''_1$ in $K(\mathcal X)$, 
where $\mathcal Y'_1 \to \mathcal Y'_1=:\mathcal Y''_1$ denotes the $e$-th iterated absolute Frobenius morphism. 
Since $\bigl( \mathcal X \times_{K(Z)} \overline{K(Z)} \bigr)_{\red}$ is a rational surface 
(Proposition \ref{p-klt-dP}), 
$\mathcal Y'_1\times_L \overline L$ is a smooth rational surface. 
Therefore, we obtain $H^i \bigl( \mathcal Y'_1, \MO_{\mathcal Y_1'} \bigr) =0$ for $i>0$.

Then there exists a commutative diagram of projective morphisms of normal varieties: 
\begin{equation}\label{e4-NWV-Omega}
\begin{CD}
Y'_1 @>\beta_1 >> Y_1\\
@VVg'_1V @VVg_1V\\
X' @>\alpha >> X\\
@VVf'V @VVfV\\
Z' @>\gamma >> Z
\end{CD}
\end{equation}
such that the horizontal arrows are finite universal homeomorphisms 
and the base change of (\ref{e4-NWV-Omega}) by $(-) \times_Z \Spec K(Z)$ is (\ref{e3-NWV-Omega}). 
Let $g:Y \to X$ be a log resolution of $(X, \Omega)$ 
which factors through $g_1:Y_1 \to X$: 
\[
g:Y \to Y_1 \xrightarrow{g_1} X.
\]
Set $Y'$ to be the normalisation of $Y$ in $K(Y' _1)$. 
Automatically, we obtain a commutative diagram of the induced morphisms: 
\begin{equation}\label{e5-NWV-Omega}
\begin{CD}
Y' @>\beta >> Y\\
@VVg'_2V @VVg_2V\\
Y'_1 @>\beta_1 >> Y_1. 
\end{CD}
\end{equation}
Combining (\ref{e4-NWV-Omega}) and (\ref{e5-NWV-Omega}), 
we obtain a commutative diagram (\ref{e2-NWV-Omega}).
By the construction, the properties (vi)--(viii) hold. 
It suffices to show (ix). 
For $\mathcal Y':=Y' \times_Z \Spec\,K(Z)$, 
$\mathfrak g'_2: \mathcal Y' \to \mathcal Y'_1$ is 
a birational morphism of projective normal surfaces over $\Spec\,K(Z)$. 
As $\mathcal Y'_1$ is smooth, $\mathcal Y'$ has at worst rational singularities 
by \cite[Proposition 1.2(2)]{Lip69} or the same argument as in Lemma \ref{l-WO-surf-go-up}. 
Since $H^i \bigl( \mathcal Y'_1, \MO_{\mathcal Y_1'} \bigr)=0$ for $i>0$, 
we have $H^i \left( \mathcal Y', \MO_{\mathcal Y'} \right)=0$ for $i>0$. 
Thus, (ix) holds. 
This completes the proof of Step \ref{s3-NWV-Omega}. 
\end{proof}

\begin{step}\label{s4-NWV-Omega} 
Assume (i)--(v) in Step \ref{s2-NWV-Omega}.
We use the same notation as in Step \ref{s3-NWV-Omega}. 
Then there exists an effective $\R$-divisor $D$ on $Y$ such that 
\begin{enumerate}
\item[(x)] $(Y, D)$ is dlt, 
\item[(xi)] the set-theoretic equation $\Nklt(Y, D)=h^{-1}(Z_1)$ holds, and  
\item[(xii)] $Rg_* \left( WI_{\Nklt \left (Y, D \right) , \Q} \right) \simeq WI_{\Nklt(X, \Omega), \Q}$. 
\end{enumerate} 
\end{step}

\begin{proof}[Proof of Step \ref{s4-NWV-Omega}] 
Let $E$ be the sum of the $g$-exceptional prime divisors $F$ such that 
$F \subset h^{-1}(Z_1)$. 
Let $E'$ be the sum of the $g$-exceptional prime divisors $F'$ such that $F' \not \subset h^{-1}(Z_1)$. 
Set 
\[
D:= g_*^{-1}\Omega^{\wedge 1}+E+(1-\epsilon)E'
\]
for sufficiently small positive real number $\epsilon$. 
It is clear that (x) and (xi) hold.

Let us prove that (xii) holds. 
We fix a closed point $x$ of $X$. 
Since the problem is local on $X$, 
it is enough to find an open neighbourhood $\widetilde{X}$ of $x \in X$ such that 
$Rg_* \left( WI_{\Nklt \left(Y, D \right), \Q} \right) \big|_{\widetilde{X}} 
\simeq WI_{\Nklt (X, \Omega), \Q}|_{\widetilde{X}}$.

By Theorem \ref{t-lc-MMP}, 
there is a $(K_Y+D)$-MMP over $X$ that terminates: 
\[
Y=Y_0  \overset{\varphi_0}{\dashrightarrow} Y_1 \overset{\varphi_0}{\dashrightarrow} 
\cdots \overset{\varphi_{\ell-1}}{\dashrightarrow} Y_{\ell}.
\]
Let $g_j:Y_j \to X$ be the induced morphism and 
let  $D_j$ be the push-forward of $D$ on $Y_j$. 
By the construction of $D$ and the $\mathbb{Q}$-factoriality of $X$, 
\begin{equation}\label{e6-NWV-Omega}
g_j^{-1} \left( \Nklt (X, \Omega) \right) = \Nklt(Y_j, D_j)
\end{equation}
holds for any $j \in \{0, ..., \ell\}$. 
In particular, it holds that $g_* \left( WI_{\Nklt \left( Y, D \right), \Q} \right) = WI_{\Nklt (X, \Omega), \Q}$. 
Thus, it is enough to prove that 
\begin{equation}\label{e7-NWV-Omega}
R^i(g_j)_* \left( WI_{\Nklt ( Y_j, \Omega_{Y_j} ) , \Q} \right) = 0
\end{equation}
for any $i>0$ and $j \in \{0, ..., \ell\}$. 
We prove this by descending induction on $j$.

We first prove (\ref{e7-NWV-Omega}) for $j=\ell$. 
We have 
\[
K_{Y}+D \sim_{X, \R} g_*^{-1}\Omega^{\wedge 1}+E+(1-\epsilon)E' - \Omega _{Y}, 
\]
where $\Omega _{Y}$ is defined by $K_Y + \Omega _Y = g^*(K_X + \Omega)$. 
Since $\epsilon$ is sufficiently small, it follows from the negativity lemma that
any $g$-exceptional prime divisor $G$ with $a_G(X, \Omega)>0$ is contracted on $Y_{\ell}$. 
Since $X$ is $\Q$-factorial, $Y_{\ell} \to X$ is isomorphic over $X \setminus \Nklt(X, \Omega)$. 
Therefore, we obtain 
\[
\Ex(g_{\ell}) \subset g_{\ell}^{-1} \left( \Nklt \left( X, \Omega \right) \right) 
= \Nklt \left( Y_{\ell}, D_{\ell} \right). 
\]
Then \cite[Proposition 2.23]{GNT} implies that the equation (\ref{e7-NWV-Omega}) holds 
when $j=\ell$.

Suppose $0 \leq j< \ell$. 
Then $\varphi_j:Y_j \dashrightarrow Y_{j+1}$ is either a divisorial contraction or a flip. 
We first treat the case when $\varphi_j$ is a divisorial contraction. 
In order to prove the equation (\ref{e7-NWV-Omega}) by induction, 
it is sufficient to show that 
\begin{align}\label{e9-NWV-Omega}
R^i (\varphi_j)_* \bigl( WI_{\Nklt \left( Y_j, D_j \right) , \Q} \bigr) = 0
\end{align}
for $i > 0$. 
Let $G = \mathrm{Ex} (\varphi_j)$ be the contracted divisor. 
Suppose $\dim \varphi_j (G) = 0$. If $G \subset \Nklt \left( Y_j, D_j \right)$, 
then the equation (\ref{e9-NWV-Omega}) follows from \cite[Proposition 2.23]{GNT}. 
If $G \not \subset \Nklt \left( Y_j, D_j \right)$, then 
$G \cap \Nklt \left( Y_j, D_j \right) = \emptyset$ holds by (\ref{e6-NWV-Omega}). 
Therefore, since $g_j$ is an isomorphism outside $G$, 
it is sufficient to check the equation (\ref{e9-NWV-Omega}) outside $\Nklt \left( Y_j, D_j \right)$, 
and this follows from Theorem \ref{t-rel-van}. 
Suppose $\dim \varphi_j (G) = 1$. In this case, the relative dimension of $\varphi$ is one, and therefore 
it is sufficient to show the equation (\ref{e9-NWV-Omega}) only for $i = 1$ (cf.\ \cite[Lemma 2.20]{GNT}). 
This follows from the exact sequence 
\[
0 \to WI_{\Nklt \left( Y_j, D_j \right) , \Q} \to W\mathcal{O}_{Y_j, \mathbb{Q}} \to 
W\mathcal{O}_{\Nklt \left( Y_j, D_j \right) , \Q} \to 0, 
\]
the fact that $R^1 (\varphi _j) _* W\mathcal{O}_{Y_j, \mathbb{Q}} = 0$ by Theorem \ref{t-rel-van}, and 
the surjectivity of $(\varphi _j) _* W\mathcal{O}_{Y_j, \mathbb{Q}} \to 
(\varphi _j) _* W\mathcal{O}_{\Nklt \left( Y_j, D_j \right) , \Q}$ 
by Theorem \ref{t-birat-connected2}.

Next we assume that $\varphi_j:Y_j \dashrightarrow Y_{j+1}$ is a flip. 
Let $\psi:Y_j \to V$ be the corresponding flipping contraction and let 
$\psi^+:Y_{j+1} \to V$ be the induced morphism. 
In order to prove the equation (\ref{e7-NWV-Omega}) by induction, 
it is sufficient to show that
\begin{equation}\label{e10-NWV-Omega}
R^i \psi_* \left( WI_{\Nklt ( Y_j, \Omega_{Y_j} ), \Q} \right)
=
R^i \psi^+ _* \left( WI_{\Nklt ( Y_{j+1}, \Omega_{Y_{j+1}} ) , \Q} \right)
\end{equation}
for any $i \geq 0$. 
When $i=0$, the equation (\ref{e10-NWV-Omega}) can be confirmed by the set-theoretic equation 
$\psi \left( \Nklt \left( Y_j, \Omega_{Y_j} \right) \right) = 
\psi ^+ \left( \Nklt \left( Y_{j+1}, \Omega_{Y_{j+1}} \right) \right)$, and this follows from 
the equation (\ref{e6-NWV-Omega}). 
In what follows, we prove that the both sides in the equation (\ref{e10-NWV-Omega}) are zero for $i \ge 1$. 
Since we work around a fixed closed point $x \in X$, 
after replacing $X$ by an open neighbourhood of $x \in X$, 
we may assume that $g_j \left( \Ex \left( \psi \right) \right)=\{ x \}$. 
There are the following two cases: 
$x \in \Nklt (X, \Omega)$ and $x \not\in \Nklt (X, \Omega)$. 
In the case when $x \in \Nklt (X, \Omega)$, 
it follows from (\ref{e6-NWV-Omega}) that $\Ex \left( \psi \right) \subset \Nklt \left( Y_j, D_j \right)$ and 
$\Ex \left( \psi^+ \right) \subset \Nklt \left( Y_{j+1}, D_{j+1} \right)$. 
Then, the both sides in the equation (\ref{e10-NWV-Omega}) are zero for $i \ge 1$ by \cite[Proposition 2.23]{GNT}. 
In the case when $x \not \in \Nklt (X, \Omega)$, we may assume that $(X, \Omega)$ is klt, 
hence so is each $\left(Y_j, D_j \right)$. 
Then the both sides in the equation (\ref{e10-NWV-Omega}) are zero for $i \ge 1$ by Theorem \ref{t-rel-van}. 
This completes the proof of Step \ref{s4-NWV-Omega}. 
\end{proof}

\begin{step}\label{s5-NWV-Omega} 
Assume (i)--(v) in Step \ref{s2-NWV-Omega}. 
Then the assertion of Proposition \ref{p-NWV-Omega} holds. 
\end{step}

\begin{proof}
We use the same notation as in Step \ref{s3-NWV-Omega} and Step \ref{s4-NWV-Omega}. 
By (vii), (ix), (xi), and \cite[Proposition 4.6.1]{CR12}, 
it holds that 
\[
R^ih_* \left( WI_{\Nklt \left(Y, D \right), \Q} \right) \big|_{Z \setminus \{z_1, \ldots, z_n\}} = 0
\]
for any $i > 0$. 
Then (xii) implies that 
\[
R^if_* \left( WI_{\Nklt (X, \Omega), \Q} \right) \big|_{Z \setminus \{z_1, \ldots, z_n\}} = 0
\]
for any $i > 0$. 
Since $X_0$ is of Fano type over $Z_0$, it follows from Theorem \ref{t-rel-van} that 
$$R^if_* \left( WI_{\Nklt \left( X, \Omega \right), \Q} \right) \big| _{Z_0} 
= R^i (f|_{X_0})_* \left( W \mathcal{O}_{X_0, \Q} \right) = 0$$ 
holds for any $i > 0$. 
Since $Z = Z_0 \cup \left( Z \setminus \{ z_1, \ldots, z_n \} \right)$, we obtain 
\[
R^if_* \left( WI_{\Nklt \left( X, \Omega \right), \Q} \right)=0.
\] 
This completes the proof of Step \ref{s5-NWV-Omega}. 
\end{proof}
Step \ref{s2-NWV-Omega} and Step \ref{s5-NWV-Omega} 
complete the proof of Proposition \ref{p-NWV-Omega}. 
\end{proof}

\begin{thm}\label{theorem:NWV}
Let $k$ be a perfect field of characteristic $p > 5$. 
Let $(X, \Delta)$ be a three-dimensional log pair over $k$ and 
let $f:X \to Z$ be a projective $k$-morphism to 
a quasi-projective $k$-scheme $Z$. 
Assume that $-(K_X + \Delta)$ is $f$-nef and $f$-big. 
Then $R^if_* \left( WI_{\Nklt \left( X, \Delta \right), \Q} \right) = 0$ for $i > 0$. 
\end{thm}

\begin{proof}
We divide the proof into two steps. 

\setcounter{step}{0}
\begin{step}\label{s1-theorem:NWV}
The assertion of Theorem \ref{theorem:NWV} holds, 
if $X$ is $\Q$-factorial and $-(K_X + \Delta)$ is $f$-ample. 
\end{step}

\begin{proof}[Proof of Step \ref{s1-theorem:NWV}]
Let $g: Y \to X$ be a projective birational morphism 
satisfying the properties (1)-(3) in Proposition \ref{p-dlt-modif}. 
Let $\Delta _Y$ be the $\mathbb{R}$-divisor on $Y$ defined by $g^*(K_X + \Delta) = K_Y + \Delta _Y$. 
Then $\left( Y, \Delta _Y ^{\wedge 1} \right)$ is dlt. 
It follows from Lemma \ref{l-dlt-modif2} that 
\[
Rg_* \left( WI_{\Nklt(Y, \Delta_Y), \Q} \right) \simeq WI_{\Nklt(X, \Delta), \Q}. 
\]
Therefore, it holds that 
\begin{equation}\label{e1-theorem:NWV}
Rh_* \left( WI_{\Nklt(Y, \Delta_Y), \Q} \right) \simeq Rf_* \left( WI_{\Nklt(X, \Delta), \Q} \right), 
\end{equation}
where $h:Y \xrightarrow{g} X \xrightarrow{f} Z$ is the composition. 
Thanks to (\ref{e1-theorem:NWV}) and Proposition \ref{p-NWV-Omega}, 
it suffices to find an effective $\mathbb{R}$-divisor $\Omega _Y$ on $Y$ such that 
\begin{enumerate}
\item[(1)] $\left( Y, \Omega _Y ^{\wedge 1} \right)$ is dlt, 
\item[(2)] $K_Y + \Omega _Y \sim _{h, \mathbb{R}} 0$, 
\item[(3)] $\Omega _Y$ is $h$-big, and
\item[(4)] $\Supp \left( \Omega _Y^{>1} \right) = \Supp \bigl( \Omega _Y ^{\ge 1} \bigr) 
= \Supp \bigl( \Delta _Y ^{\ge 1} \bigr)$. 

\end{enumerate}

Since $X$ is $\Q$-factorial, 
there exists an effective $\mathbb{R}$-divisor $F$ on $Y$ such that $-F$ is $g$-ample and 
$\Supp F = \Ex (g)$. 
Since $-(K_Y + \Delta _Y)$ is the pullback of an $f$-ample $\mathbb{R}$-divisor 
$-(K_X + \Delta)$ on $X$, 
it follows that $-(K_Y + \Delta _Y)- \epsilon F$ is $h$-ample for any sufficiently small $\epsilon > 0$. 

Note that $\Supp F \subset \Supp \bigl( \Delta _Y ^{\ge 1} \bigr)$. 
Thus, we can find an effective $\mathbb{R}$-divisor $B$ on $Y$ 
such that $B \ge \epsilon F$, $-(K_Y + \Delta _Y) - B$ is $h$-ample and 
$\Supp B = \Supp \bigl( \Delta _Y ^{\ge 1} \bigr)$. 
Lemma \ref{l-perturb-dlt} enables us to find an effective $\mathbb{R}$-divisor $A$ on $Y$ 
such that $A \sim_{h, \mathbb{R}} -(K_Y + \Delta _Y)- B$ 
and $\left( Y, \Delta _Y ^{\wedge 1} + 2A \right)$ is dlt. 
In particular, $\left( Y, \Delta _Y ^{\wedge 1} + A \right)$ is dlt. 

Set $\Omega _Y :=  \Delta _Y + A + B$. 
Then both (2) and (3) hold automatically. 
We obtain 
\begin{equation}\label{e2-theorem:NWV}
\Omega _Y ^{\wedge 1}= \left( \Delta _Y + A + B \right)^{\wedge 1}
= \left( \Delta _Y + A \right)^{\wedge 1}= \Delta _Y ^{\wedge 1} + A,
\end{equation}
where the second equality follows from 
$\Supp B = \Supp \bigl( \Delta _Y ^{\ge 1} \bigr)$ and the third one holds by the fact that 
$\left( Y, \Delta _Y ^{\wedge 1} + 2A \right)$ is dlt. 
Thus (1) holds. 

Let us prove (4). It is clear that 
$\Supp \left( \Omega _Y^{>1} \right)  \subset  \Supp \bigl( \Omega _Y ^{\ge 1} \bigr)$. 
The inverse inclusion follows from 
\[
\Supp \bigl( \Omega _Y ^{\ge 1} \bigr) 
= \Supp \left( \Delta _Y + A + B \right)^{\ge 1} 
= \Supp \left( \Delta _Y + B \right)^{\ge 1} 
\]
\[
= \Supp ( \Delta _Y + B)^{>1} 
\subset \Supp (\Delta _Y +A+ B)^{>1} 
=\Supp \left( \Omega _Y ^{> 1} \right),
\]
where the second equality follows from the fact that $\left( Y, \Delta _Y ^{\wedge 1} + 2A \right)$ is dlt, and 
the third one holds by $\Supp B = \Supp \bigl( \Delta _Y ^{\ge 1} \bigr)$. 
Then we obtain the remaining equality as follows: 
\[
\Supp \Omega _Y ^{\geq 1} = \Supp \left( \Omega _Y ^{\wedge 1} \right)^{= 1}
=\Supp \left( \Delta _Y ^{\wedge 1} + A \right) ^{= 1}= \Supp \Delta _Y ^{\geq 1}, 
\]
where the second equality follows from (\ref{e2-theorem:NWV}) and 
the third equality holds by the fact that $\left( Y, \Delta _Y ^{\wedge 1} + 2A \right)$ is dlt. 
Thus (4) holds. 
This completes the proof of Step \ref{s1-theorem:NWV}. 
\end{proof}

\begin{step}\label{s2-theorem:NWV}
The assertion of Theorem \ref{theorem:NWV} holds without any additional assumptions. 
\end{step}

\begin{proof}[Proof of Step \ref{s2-theorem:NWV}] 
Lemma \ref{l-dlt-modif2} enables us to replace $(X, \Delta)$ by 
the log pair $(Y, \Delta_Y)$ appearing in Proposition \ref{p-dlt-modif} (cf.\ (\ref{e1-theorem:NWV})). 
Hence, we may assume that $X$ is $\Q$-factorial. 

Since $-(K_X + \Delta)$ is $f$-nef and $f$-big, 
there exists an effective $\mathbb{R}$-divisor $E$ such that 
$-(K_X + \Delta) -  \epsilon E$ is $f$-ample for any real number $\epsilon$ 
satisfying $0<\epsilon<1$. 
The equation 
$$\mathrm{Nklt}(X, \Delta) = \mathrm{Nklt}(X, \Delta + \epsilon E)$$ 
holds for sufficiently small $\epsilon>0$. 
Therefore, replacing $\Delta$ by $\Delta + \epsilon E$, 
we may assume that $-(K_X + \Delta)$ is $f$-ample. 
Hence, Step \ref{s2-theorem:NWV} follows from Step \ref{s1-theorem:NWV}.
\end{proof}
Step \ref{s2-theorem:NWV} completes the proof of Theorem \ref{theorem:NWV}. 
\end{proof}

\begin{thm}\label{theorem:NWV2}
Let $k$ be a perfect field of characteristic $p > 5$. 
Let $(X, \Delta)$ be a three-dimensional log pair over $k$ and 
let $f:X \to Z$ be a projective $k$-morphism to a quasi-projective $k$-scheme $Z$. 
Let $Z'$ be a closed subscheme of $Z$, set $X':=X \times_Z Z'$, 
and let $f':X' \to Z'$ be the induced morphism. 
Assume that $-(K_X + \Delta)$ is $f$-nef and $f$-big. 
Then the following hold. 
\begin{enumerate}
\item 
$R^if_* \left( WI_{X' \cup \Nklt(X, \Delta), \Q} \right)=0$ for $i>0$. 
\item 
$R^if'_* \left( WI_{j^{-1}(\Nklt (X, \Delta)), \Q} \right) = 0$ 
for $i > 0$, where $j:X' \to X$ denotes the induced closed immersion. 
\end{enumerate}
\end{thm}

\begin{proof} 
Taking the Stein factorisation of $f$, we may assume that $f_*\MO_X=\MO_Z$. 

Let us prove (1). 
Since the problem is local on $Z$, 
the problem is reduced to the case when $Z$ is affine. Hence we can write 
\[
Z=\Spec A, \quad Z'=\Spec (A/I),\quad I= (a_1, \ldots, a_r)
\]
for some $a_1, \ldots, a_r \in A$. 
We show the assertion of (1) by induction on $r$. 

We now treat the case when $r=1$. 
If $Z'=Z$, then there is nothing to show. 
If $Z'=\emptyset$, then the assertion follows from Theorem \ref{theorem:NWV}. 
Thus, we may assume that $X'$ is a nonzero effective Cartier divisor on $X$. 
Since $- \left( K_X+\Delta+X' \right)$ is $f$-nef and $f$-big, 
Theorem \ref{theorem:NWV} implies that 
\[
R^if_* \left( WI_{X' \cup \Nklt (X, \Delta), \Q} \right) = R^if_* \left(WI_{\Nklt \left( X, \Delta+X' \right), \Q} \right)=0
\]
for $i>0$. 
Thus the assertion of (1) holds if $r=1$. 

Assume that $r \geq 2$ and 
that the assertion of (1) holds for the case when $I$ is generated by less than $r$ elements. 
We set 
\begin{align*}
Z'':=&\Spec \left( A/ \left( f_1, \ldots, f_{r-1} \right) \right), \quad Z_r:=\Spec \left( A/(f_r) \right), \\
&X'':=X \times_Z Z'', \quad X_r:=X \times_Z Z_r.
\end{align*}
For $N:=\Nklt (X, \Delta)$, 
we have the exact sequence (Lemma \ref{l-MV}) 
\[
0 \to WI_{ \left( X'' \cup N \right) \cup \left( X_r \cup N \right), \Q} 
\to WI_{X'' \cup N, \Q} \oplus WI_{X_r \cup N, \Q} 
\to WI_{X' \cup N, \Q} 
\to 0.
\]
Then the assertion (1) follows by induction. 

Let us prove (2). 
Thanks to the exact sequence (Lemma \ref{l-MV})  
\[
0 \to WI_{\Nklt (X, \Delta)\cup X', \Q} \to 
WI_{\Nklt (X, \Delta), \Q} \oplus WI_{X', \Q} 
\to WI_{\Nklt (X, \Delta)\cap X', \Q} \to 0, 
\]
(1) implies that the natural map 
\[
\varphi^i:R^if_* \left( WI_{X', \Q} \right) \to R^if_* \left (WI_{\Nklt (X, \Delta)\cap X', \Q} \right)
\]
is bijective for $i>0$. 
Thanks to the following commutative diagram with exact horizontal sequence: 
{\small 
\[
\begin{CD}
0 @>>> WI_{X', \Q} @>>> W\MO_{X, \Q} @>>> W\MO_{X', \Q} @>>> 0\\
@. @VV\varphi V @| @VVV\\
0 @>>> WI_{\Nklt (X, \Delta)\cap X', \Q} @>>> W\MO_{X, \Q} @>>> W\MO_{\Nklt (X, \Delta)\cap X', \Q} @>>> 0,
\end{CD}
\]
}

\noindent
the snake lemma induces the exact sequence
\[
0 \to WI_{X', \Q} \xrightarrow{\varphi} WI_{\Nklt (X, \Delta)\cap X', \Q} \to j_*WI_{j^{-1}\left( \Nklt (X, \Delta) \right), \Q} \to 0. 
\]
Since the map $\varphi^i:R^if_* \left( WI_{X', \Q} \right) \to R^if_* \left( WI_{\Nklt (X, \Delta)\cap X', \Q} \right)$ 
induced by $\varphi$ is bijective for $i>0$, we get 
\begin{equation}\label{e1-theorem:NWV2}
Rf_* \left( j_*WI_{j^{-1} \left( \Nklt (X, \Delta) \right), \Q} \right) \simeq f_*j_*WI_{j^{-1} \left( \Nklt (X, \Delta) \right), \Q}. 
\end{equation}
If $i$ denotes the induced closed immersion $Z' \to Z$, then it holds that 
\begin{align*}
f_*j_*WI_{j^{-1}\left( \Nklt (X, \Delta) \right), \Q} 
&\simeq Rf_* \left( j_*WI_{j^{-1} \left( \Nklt (X, \Delta) \right), \Q} \right)\\
&\simeq Rf_*Rj_*WI_{j^{-1} \left( \Nklt (X, \Delta) \right), \Q}\\
&\simeq Ri_*Rf'_*WI_{j^{-1} \left( \Nklt (X, \Delta) \right), \Q}\\
&\simeq i_*Rf'_*WI_{j^{-1} \left( \Nklt (X, \Delta) \right), \Q},
\end{align*}
where the first isomorphism follows from (\ref{e1-theorem:NWV2}), and 
the second and last ones hold because $j_*$ and $i_*$ are exact functors. 
This completes the proof of (2). 
\end{proof}

\begin{thm}\label{theorem:conn}
Let $k$ be a perfect field of characteristic $p > 5$. 
Let $(X, \Delta)$ be a three-dimensional log pair over $k$ and 
let $f:X \to Z$ be a projective $k$-morphism to a quasi-projective $k$-scheme $Z$ such that $f$ has connected fibres. 
Assume that $-(K_X + \Delta)$ is $f$-nef and $f$-big. 
Then the induced morphism $\Nklt (X, \Delta) \to Z$ 
has connected fibres. 
\end{thm}

\begin{proof}
Thanks to Theorem \ref{theorem:NWV}, the homomorphism 
$$W\MO_{Z, \Q}=f_* W\MO_{X, \Q} \to f_* \left( W\MO_{\Nklt (X, \Delta), \Q} \right)$$
is surjective. 
Since this homomorphism factors through $W\MO_{f \left( \Nklt (X, \Delta) \right), \Q}$, 
also the induced homomorphism 
$$\theta:W\MO_{f \left( \Nklt (X, \Delta) \right), \Q} \to f_* \left( W\MO_{\Nklt (X, \Delta), \Q} \right)$$
is surjective. Since this is automatically injective, $\theta$ is bijective. 
Therefore, the morphism $\Nklt (X, \Delta) \to f \left( \Nklt (X, \Delta) \right)$ has connected fibres (cf.\ Lemma \ref{l-GNT2.22}), as desired. 
\end{proof}

\section{Application to rational points on varieties over finite fields}

As an application of our vanishing theorem of Nadel type 
(Theorem \ref{theorem:NWV}, Theorem \ref{theorem:NWV2}), 
we deduce some consequences for rational points on varieties over finite fields 
(Theorem \ref{t-rat-pt1}, Theorem \ref{t-rat-pt2}).

\begin{thm}\label{t-rat-pt1}
Let $(X, \Delta)$ be a three-dimensional log pair over 
a finite field $k$ of characteristic $p>5$. 
Let $f:X \to Y$ be a projective $k$-morphism to a quasi-projective $k$-scheme $Y$ 
such that $f_*\MO_X=\MO_Y$. 
Assume that $-(K_X+\Delta)$ is $f$-nef and $f$-big. 
Then, the congruence 
\[
\# X(k) -  \# V (k) \equiv \# Y(k) - \# f(V) (k)  \mod \#k
\]
holds, where $V := \mathrm{Nklt}(X, \Delta)$. 
\end{thm}

\begin{proof}
If $Y(k)=\emptyset$, then there is nothing to show. 
Thus the problem is reduced to the case when $Y(k)\neq \emptyset$. 
Fix a $k$-rational point $y \in Y$. 
Since the problem is local on $Y$, we may assume that $Y(k)=\{y\}$. 
If $\Nklt(X, \Delta) \cap f^{-1}(y)=\emptyset$, then the assertion follows from 
\cite[Theorem 5.4]{GNT}. 
Hence, we may assume that $y \in f \left( \Nklt(X, \Delta) \right)=f(V)$, i.e.\ $V_y \neq \emptyset$. 
In particular, it holds that 
\[
Y(k) =f(V)(k) =\{y\}.
\]
Since 
\[
X(k)=X_y(k), \quad V(k)=V_y(k),
\]
it suffices to prove that 
\[
\# X_y(k) \equiv \# V_y (k)  \mod \# k.
\]
Consider the exact sequence 
\[
0 \to WI_{V_y, \Q} \to W\MO_{X_y, \Q} \to W\MO_{V_y, \Q} \to 0.
\]
Thanks to Theorem \ref{theorem:NWV2} (2), it holds that 
\[
H^i \left( X_y, WI_{V_y, \Q} \right)=0
\]
for $i>0$. 
Furthermore, we get $H^0 \left( X_y, WI_{V_y, \Q} \right)=0$ by  $V_y \neq \emptyset$ 
and the fact that $X_y$ is connected. 
Hence, the natural map 
\[
H^i \left( X_y, W\MO_{X_y, \Q} \right) \to H^i \left( V_y, W\MO_{V_y, \Q} \right)
\]
is bijective for $i\geq 0$. 
Therefore the assertion holds by \cite[Proposition 6.9 (i)]{BBE07}. 
\end{proof}

\begin{cor}\label{cor:klt_locus}
Let $(X, \Delta)$ be a three-dimensional geometrically connected projective log pair over a finite field $k$ of characteristic $p>5$. 
Assume that $-(K_X + \Delta)$ is nef and big and that $(X, \Delta)$ is not klt. 
Then, the congruence 
\[
\# X (k) \equiv \# V (k)  \mod \# k
\]
holds, where $V := \mathrm{Nklt}(X, \Delta)$. 
\end{cor}
\begin{proof}
Applying Theorem \ref{t-rat-pt1} for $Y := \Spec k$, the assertion holds. 
\end{proof}

\begin{thm}\label{t-rat-pt2}
Let $k$ be a perfect field of characteristic $p>5$. 
Let $X$ be a projective normal variety over $k$ with $\dim X\leq 3$. 
Let $D$ be a nonzero effective $\Q$-Cartier $\Z$-divisor on $X$. 
Assume that there exists an effective $\mathbb{R}$-divisor $\Delta$ such that 
\begin{enumerate}
\item[(a)] $(X, \Delta)$ is klt, 
\item[(b)] $-(K_X+\Delta)$ is nef and big, and 
\item[(c)] $-(K_X+\Delta+D)$ is nef and big. 
\end{enumerate}
Then the following hold. 
\begin{enumerate}
\item 
The equation 
\[
H^i(D, W\MO_{D, \Q})=0
\]
holds for $i>0$, and the induced map 
\[
H^0 \left( X, W\MO_{X, \Q} \right) \to H^0 \left( D, W\MO_{D, \Q} \right)
\]
is bijective. 
\item 
If $k$ is a finite field and $X$ is geometrically connected over $k$, 
then the congruence  
\[
\# D(k) \equiv 1 \mod \# k
\]
holds. 
\end{enumerate}
\end{thm}

\begin{proof}
Let us prove (1). 
We have the exact sequence 
\[
0 \to WI_{D, \Q} \to W\MO_{X, \Q} \to W\MO_{D, \Q} \to 0.
\]
It follows from (a), (b) and Theorem \ref{t-gnt} (2) that the equation 
\[
H^i \left( X, W\MO_{X, \Q} \right)=0
\]
holds for $i>0$. 
Note that $\Nklt(X, \Delta+D)=\Supp D$. 
Hence, by (c) and Theorem \ref{theorem:NWV}, we get 
\[
H^i \left( X, WI_{D, \Q} \right)=0
\]
for $i>0$. 
Since $D \neq 0$, we obtain $H^0 \left( X, WI_{D, \Q} \right)=0$. 
This completes the proof of (1).

Let us show (2). 
Thanks to (1), we may apply \cite[Proposition 6.9 (i)]{BBE07} and 
obtain the congruence: $\# X(k) \equiv \# D(k) \mod \# k$. 
On the other hand, we have another congruence 
$\# X(k) \equiv 1 \mod \# k$, 
which is guaranteed by \cite[Theorem 5.4]{GNT}. 
To summarise, we get $\# D(k) \equiv 1 \mod \# k$, as desired. 
\end{proof}

\begin{rem}
As applications of a vanishing theorem of Nadel type (Theorem \ref{theorem:NWV}), 
we obtain two results: 
the Koll\'ar--Shokurov connectedness theorem (Theorem \ref{theorem:conn}) and the existence of rational points (Theorem \ref{t-rat-pt1}). 
For certain special cases, these two consequences are also related as follows. 

Let $k$ be a perfect field of characteristic $p>5$ and 
let $(X, \Delta)$ be a projective log pair over $k$ with $\dim X\leq 3$ 
such that $-(K_X+\Delta)$ is $f$-nef and $f$-big. 
Assume that $X$ is geometrically connected over $k$, 
$\Nklt(X, \Delta) \neq \emptyset$ and $\dim \Nklt(X, \Delta)=0$. 
Then $\Nklt(X, \Delta)$ is geometrically connected over $k$ by the Koll\'ar--Shokurov connectedness theorem (Theorem \ref{theorem:conn}), 
which implies that $\Nklt(X, \Delta)$ consists of a single $k$-rational point. 
\end{rem}


\end{document}